\documentclass[reqno]{amsart}
\usepackage[utf8]{inputenc}
\usepackage{enumitem}

\usepackage{amsmath,amssymb,amsbsy,amsfonts,amsthm,latexsym,amsopn,amstext,mathtools,leftidx,
            amsxtra,euscript,amscd,verbatim,epsf,colordvi,graphics}
\usepackage{color}
\usepackage{pgf,tikz}
\usetikzlibrary{decorations.pathreplacing}
\usetikzlibrary{arrows}
\usepackage{mathrsfs}
\usepackage{graphicx}

\usepackage{hyperref}
\usepackage{cleveref}
\usepackage[margin=11pt,font=footnotesize,labelfont=bf]{caption}
\newtheoremstyle{tplain}{3pt}{3pt}{\rmfamily}{}{\bfseries}{.}{0.5em}{}
\theoremstyle{tplain}

\usepackage{pgf,tikz}
\usepackage{tcolorbox}
\usepackage[enableskew]{youngtab}
\definecolor{darkgreen}{cmyk}{1,0,1,0}
\newtheorem{thm}{Theorem}
\newtheorem{lem}{Lemma}
\newtheorem{ex}{Example}
\newtheorem{cor}{Corollary}
\newtheorem{prop}{Proposition}
\newtheorem{obs}{Remark}
\newtheorem{defi}{Definition}

\newcommand{\blue}{\color{blue}}
\newcommand{\red}{\color{red}}
\newcommand{\green}{\color{green}}

\makeatletter
\newcommand*\bigcdot{\mathpalette\bigcdot@{.5}}
\newcommand*\bigcdot@[2]{\mathbin{\vcenter{\hbox{\scalebox{#2}{$\m@th#1\bullet$}}}}}
\makeatother

\newcommand{\YT}[3]{
\vcenter{\hbox{
\begin{tikzpicture}[x={(0in,-#1)},y={(#1,0in)}] 
\foreach \rowi [count=\i] in {#3} {
 \foreach \e [count=\j] in \rowi {
  \draw (\i,\j) rectangle +(-1,-1);
  \draw (\i-0.5,\j-0.5) node {$#2\e$};
 }
}
\end{tikzpicture}
}}
}

\title[The symplectic left companion  of a  LR-Sundaram tableau and the Kwon condition] {The symplectic left companion  of a  Littlewood-Richardson-Sundaram  tableau \\and the Kwon property}

\author{Olga Azenhas}
\address{ University of Coimbra, CMUC, Department of Mathematics, Portugal}
\email{oazenhas@mat.uc.pt}

\keywords{Kwon and Sundaram branching model, Littlewood-Richardson-Sundaram  tableaux, left companion.}

\subjclass[2000]{05E05, 05E10, 05E14, 17B37, 68Q17}

\begin{document}

\begin{abstract} As a consequence of the Littlewood-Richardson (LR) commuters coincidence and the Kumar-Torres branching model  via  Kushwaha-Raghavan-Viswanath flagged hives, we have solved the Lecouvey- -Lenart conjecture on the  bijections between the Kwon and Sundaram
  branching models for the pair $({GL}_{2n}(\mathbb{C}), {Sp}_{2n}(\mathbb{C})) $ consisting of  the  general linear group ${GL}_{2n}(\mathbb{C})$
and the symplectic group ${Sp}_{2n}(\mathbb{C})$. In particular, thanks to the Henriques-Kamnitzer $gl_n$-crystal commuter, we have   recognized that the left companion of an LR-Sundaram tableau is characterized by the   Kwon  symplectic condition. We now use the construction of the left Gelfand-Tsetlin pattern, or left companion tableau, of an LR-Sundaram tableau to exhibit the Kwon symplectic property, a mirror of the flag on its right companion tableau. This is equivalent to the  restriction of the Berenstein-Gelfand-Zelevinsky LR model, on the interpretation of Gelfand-Tsetlin patterns, to symplectic Gelfand-Tsetlin patterns as left companions of LR-Sundaram tableaux.
\end{abstract}
\maketitle

\tableofcontents

\section{Introduction}
Consider $G$  a group and $\hat G$ a complete set of representatives of the
equivalence classes of certain irreducible $G$-modules \cite{watanabe}. Given  $H$  a subgroup of $G$, a natural and interesting problem
is to determine how and if a given irreducible $G$-module $V \in \hat G$ decomposes
into irreducible $H$-submodules  \cite{littlewood,sundaram,HTW,kwon18,watanabe,sathishtorres}:
\begin{align}V \simeq \bigoplus_{W\in \hat H}
W^{c^V_{W}}.\nonumber
\end{align}
The multiplicity  numbers $c^V_{W}$ are called \emph{branching coefficients} of the pair $(G,H)$.
An explicit description of the branching coefficients is called a \emph{branching rule} for the pair $(G,H)$. For example, the Littlewood-Richardson rule \cite{LR} gives a branching rule for the pair $({GL}_{m}(\mathbb{C})\times {GL}_{m}(\mathbb{C}), {GL}_{m}(\mathbb{C}))$.
In this note one considers the pair $(G,H)=(GL_{2n}(\mathbb{C}),Sp_{2n}(\mathbb{C}))$. The polynomial irreducible representations of $GL_{2n}(\mathbb{C})$ respectively $Sp_{2n}(\mathbb{C})$ are parameterized by partitions $\lambda$ of length $\le 2n$ respectively partitions $\mu$ of $\le n$, and we denote the corresponding  branching coefficient by $c^\lambda_\mu$.

Littlewood \cite{littlewood} has given a branching rule only in the case where both partitions have length $\le n$. Sundaram has given a complete branching rule \cite{sundaram, sundaram90} by counting certain Littlewood-Richardson (LR) tableaux, called  Littlewood-Richardson-Sundaram tableaux in \cite{leclen,sathishtorres}, and symplectic LR tableaux in \cite{watanabe}.
Schumann and Torres \cite{schumanntorres} proved a conjectural branching rule by Naito and Sagaki \cite{naitosagaki} in terms of Littelmann
paths.

Kwon \cite{kwon18} has provided branching rules for various pairs in particular the pair $(GL_{2n}(\mathbb{C}),Sp_{2n}(\mathbb{C}))$. To express the branching coefficients $c^\lambda_\mu$ he enumerates certain sets in  its
combinatorial spinor model \cite{kwon} for crystals of classical type. As expected   the problem of comparing the Sundaram and the Kwon branching rules  addresses. This has been considered by Lecouvey and Lenart \cite{leclen} by conjecturing an explicit bijection  between the two models via the combinatorial $R$-matrix realized by the Henriques-Kamnitzer LR commuter \cite{HK2,HenKam} which has several realizations \cite{knutson, fultonbuch,tyong2,akt16, az18v5,tka18, azkoma25,az18v5} depending also on the LR model.
Recently Kumar and Torres \cite{sathishtorres} use the LR commuter
  by Kushwaha–Raghavan–Viswanath for flagged hives \cite{krv21,krsv24} to establish a bijection between the Sundaram and Kwon branching models.
As conjectured by Pack and Vallejo \cite{pakvallejo} and proved in \cite{pakvallejo,DK08,akt16, azkoma25, az18v5} all these realizations coincide and henceforth the Lecouvey-Lenart conjecture \cite{leclen} is solved.

As  one shows in \cite{az18v5} the objects counted by the Kwon branching model, after a rephrasing of that model by Lecouvey-Lenart, are precisely the left companions of  LR-Sundaram tableaux. While Kumar-Torres use the right companion tableau of an  LR-Sundaram tableau in their bijection via flagged hives,  here, we use the left companion to  directly  show that   the Sundaram flag condition on an LR tableau   is mirrored on the left companion as a  symplectic Kwon property. In other words, the construction of the left companion tableau of an LR-Sundaram tableau gives a symplectic semistandard tableau, equivalently, a semistandard tableau satisfying the symplectic Kwon property. The symplectic condition on the left companion parallels the flag on   right companion tableau \cite{sathishtorres} of an LR-Sundaram tableau as in \cite{sathishtorres}. Indeed the interlocking of the companion Gelgand-Tsetlin pattern pair of an LR-Sundaram tableau, consisting of the left and the right companions in the form of Gelfand-Tsetlin patterns  \cite{gt50},
leads to a flagged hive.

 Right and left companions of an LR tableau    have  their origin in  the form of Gelfand-Tsetlin patterns \cite{gt50} in the works of Gelfand-Zelevinsky \cite{GZ85, gzpolyed} and Berenstein-Zelevinsky \cite{BZphy} in their interpretation of Gelfand-Tsetlin patterns as  LR models.  For the one-to-one correspondence between Gelfand-Tsetlin patterns and semistandard tableaux and applications to mathematical-physics, see for instance, Louck \cite{Louck}. (See also
 \cite{dookim,  akt16} and \cite[Appendix]{Nak05} for a crystal theoretic interpretation of the LR rule in terms of right companions.)  For the construction of  Gelfand-Tsetlin bases  for classical Lie algebras where instead of
semistandard tableaux, the basis vectors are parameterized by the combinatorial objects
called  Gelfand–Tsetlin patterns, see Molev's review  \cite{molev06}. It suffices here to say that the branching rule for the pair $(\mathfrak{gl}_m,\mathfrak{gl}_{m-1})$, the restriction of the general Lie algebra $\mathfrak{gl}_m$ to the subalgebra $\mathfrak{gl}_{m-1}$, is given by the between-ness conditions of their highest weights $\lambda^{(m)}$ and $\lambda^{(m-1)}$ which means that the skew diagram $\lambda^{(m)}/\lambda^{(m-1)}$ is a horizontal strip. These are  two rows of the Gelfand-Tsetlin pattern of type $\lambda=\lambda^{(m)}$ defined by the chain of embeddings $\mathfrak{gl}_m\supseteq\mathfrak{gl}_{m-1})\supseteq\cdots \supseteq\mathfrak{gl}_{1}$. We thus intertwine  semistandard tableaux with Gelfand-Tsetlin patterns.

To summarize our contribution we exhibit the branching models of Kwon and Sundaram via the left companion of a Littlewood-Richardson-Sundaram (LRS) tableau. In other words, we show that the Berenstein-Gelfand-Zelevinsky LR model, on the interpretation of Gelfand-Tsetlin patterns, restrict to symplectic Gelfand-Tsetlin patterns as left companions of LR-Sundaram tableaux.

More precisely, for partitions $\mu\subset \lambda$ with $\mu$ of length at most $n$ and $\lambda$ of length at most $2n$, the Sundaram branching rule says:
the branching coefficient $c_\mu^\lambda$
equals the cardinality
of the set
$$
LRS(\lambda,\mu):=\bigcup_\nu
LRS(\lambda/\mu,\nu)$$
where the union is taken over all even partitions $\nu$, and $LRS(\lambda/\mu,\nu)$ denotes the set consisting of the Littlewood-Richardson-Sundaram tableaux of shape $\lambda/\mu$ and weight the even partition $\nu$.

According  to a reformulation of the Kwon’s branching rule by Lecouvey–Lenart \cite[Section 8]{leclen} (see also \cite{sathishtorres}),
the branching coefficient
$c^\lambda_\mu$
equals the cardinality of the set
$$
LRK(\lambda,\mu):=\bigcup_\nu
LRK^\lambda_{\nu,\mu}$$
where the union is taken over all even partitions $\nu$, and $LRK^\lambda_{\nu,\mu}$ denotes the set of right companions $T$ (of shape $\mu$) of the  LR tableaux in $LR(\lambda/\nu,\mu)$ (note  $\nu$ and $\mu$ are swapped) whose evacuation (or Sch\"utzenberger involution) $S(T)$  satisfy the Kwon property (also called symplectic property). Those tableaux $S(T)$ are precisely the left companions of the set Littlewood-Richardson-Sundaram tableaux $LRS(\lambda/\mu,\nu)$.

Our main result is stated as follows and illustrated in Examples \ref{ex:example}. It describes how the Sundaram property violation mirrors on its left companion as a symplectic condition violation.

{\textbf{Main Theorem} \ref{th:main} Let $T\in LR(\lambda/\mu,\nu)$ with $\nu$ even, on the alphabet $[2n]$, and $ \ell(\mu)\le n$.
 $T$ does not satisfy the Sundaram property if and only if $G_\mu(T)$, the left companion of $T$, is not symplectic.

Moreover,  in this case, there exists a unique $t\ge 0$ such that the following are equivalent
\begin{enumerate}

\item  ${n+t+1}$  is  the minimal row number of $T$ where the  Sundaram property  violation occurs.

\item $T(n+1,1)\ge 2,\, T(n+2,1)\ge 4,\dots, T(n+t+1-2,1)\ge 2(t+1-2)$ and
  \begin{align*}T(n+t,1)=2t,\; T(n+t+1,1)= 2t+1,\; T(n+t+2,1)=2(t+1).\qquad\qquad\qquad\qquad\qquad\quad 
 \end{align*}
 \item  the maximal row    of $G_\mu(T)$ where a symplectic violation  occurs is among the bottom most $t+1$ cells  $(\ell(\mu),1)$,  $(\ell(\mu)-1,1),\dots$, or $ (\ell(\mu)-t,1) $ of the first column of $G_\mu(T)$.
\end{enumerate}

As a corollary, one obtains  the restriction of the Berenstein-Gelfand-Zelevinsky LR model on the interpretation of the left Gelfand-Tsetlin patterns \cite[Theorem 4.3]{BZphy}, to symplectic left Gelfand-Tsetlin patterns as left companions of LR-Sundaram tableaux.

{\textbf{Corollary} The Berenstein-Gelfand-Zelevinsky LR model on  Gelfand-Tsetlin patterns in \cite[Theorem]{BZphy} restricts to symplectic left Gelfand-Tsetlin patterns as left companions of LR-Sundaram tableuax. In other words, the Berenstein-Gelfand-Zelevinsky LR model in \cite[Theorem 4.3]{BZphy} restricted to symplectic Gelfand-Tselin patterns is in bijection with symplectic LR  (or LR-Sundaram) tableaux.
}
\medskip

Watanabe \cite{watanabe} has also recently    established a new branching rule for the pair   $(GL_{2n}(\mathbb{C}), Sp_{2n}(\mathbb{C}))$. It is an
interesting question to ask  how the Watanabe branching rule bijects  to Sundaram or Kwon branching rules.

This paper is organized in four sections. Section \ref{sec:preliminaries0} introduces the relevant notation and how semistandard tableaux and Gelfand-Tsetlin patterns are in natural bijection.  Section \ref{sec:symplectic} introduces semistandard symplectic tableaux (King tableaux rephrased in the alphabet $[2n]$) \cite{watanabe} also called tableaux satisfying the Kwon property \cite{sathishtorres} and define the corresponding symplectic Gelfand-Tsetlin patterns in Definition \ref{def:sympgt}. The left companion of an LR tableau and Littlewood-Richardson-Sundaram tableaux are defined in Section \ref{sec:lrsymplectic}. Left Gelfand-Tsetlin patterns are left companions of LR tableaux and constitute a Berenstein-Gelfand-Zelevinsky LR model \cite{BZphy}. The main result and examples also appear in Section \ref{sec:lrsymplectic}.

\section*{Acknowledgements}
The author acknowledges financial support by the Centre for Mathematics of the University of Coimbra (CMUC, https://doi.org/10.54499/UID/00324/2025) under the Portuguese Foundation for Science and Technology (FCT), Grants UID/00324/2025 and UID/PRR/00324/2025.

\section{Preliminaries}\label{sec:preliminaries0}

\subsection{Partitions,  semistandard tableaux and Gelfand-Tsetlin patterns}\label{sec:preliminaries}
A \emph{partition} is a weakly decreasing sequence of nonnegative integers $\gamma = \gamma_1\ge \gamma_2 \ge \cdots$
such that $\gamma_k =0$ for some $k \ge1$. The \emph{length} of $\gamma$ is the maximal $i$ such that $\gamma_i> 0$ also called the \emph{number}
of \emph{parts} or \emph{length} of $\gamma$.  We write the partition $\gamma$ as a vector
 $\gamma = (\gamma_1, \gamma_2, \dots,\gamma_k)$ for $k \ge  \ell(\gamma)$. A partition $\gamma$ is identified
with its \emph{Young diagram} $D(\gamma)$ which is a left and top justified collection of boxes (or cells)
with $\gamma_k$ many boxes in the $k$th row for all $k \in \mathbb{Z}_> 0$. In particular, the empty Young diagram and  the  partition $(0)$ are identified. The number of cells of $D(\gamma)$ is  the sum of the parts of $\gamma$ and is denoted by $|\gamma|$.
The boxes or cells of the Young diagram of $\gamma$ are identified by its  coordinates $(i,j)$ in the matrix style, that is, $1\le i\le \ell(\gamma)$ and $1\le j\le \gamma_i$.

Let $\gamma$, $\mu$  partitions   with $\mu \subset \gamma$, that is,  $\mu_i\le  \gamma_i$ for all $ i\in\mathbb{Z}_>0 $, or the Young diagram of $\mu$ is a subset of
the Young diagram of $\gamma$. A \emph{semistandard
tableau} $T$ of (skew) shape $\gamma/\mu$ is a map (or a filling of $D(\gamma)$)
 $$T:D(\gamma)\rightarrow \mathbb{Z}_\ge0, \; T(i,j)\mapsto T(i,j),$$
assigning a nonnegative integer to each box of $\gamma$  such that it is weakly increasing as we go from left to right along a row and strictly
increasing as we go from top to bottom along a column, and equal to $0$  in the boxes corresponding to $\mu$,
\begin{align}&T(i,j)\le T(i,j+1),\; T(i,j)<T(i+1,j), \mbox{ for all $(i,j)\in D(\gamma)\setminus D(\mu)$,}\nonumber\\
&\mbox{ and }  T(i,j)=0, \mbox{ for $(i,j)\in  D(\mu)$,}
\nonumber
\end{align}
where we set $T(a,b):=\infty$ if $(a, b) \notin D(\gamma)$.
 Usually $T(i,j)$ is just referred as
the entry in the box $(i,j)$ and we omit the zeroes  in the boxes of $\mu$. A positive integer
$m\ge \ell(\gamma)$ will be fixed and $[0, m]:=\{0, 1, \dots, m\}$ will be used as a co-domain for map $T$. We call
$[m]:=\{1, \dots, m\}$ the alphabet of the semistandard tableau $T$.
 In this case, we will denote the set of semistandard tableaux of shape $\gamma/\mu$ by
$SST_k(\gamma/\mu)$. When $\mu=(0)$, we just write $SST_m(\gamma)$. The \emph{weight} or \emph{content} of $T$ is the nonnegative vector $(\alpha_1,\dots,  \alpha_m)$, where $\alpha_i:=\#\{(a,b)\in D(\gamma):T(a,b)=i\}$ for $i\in[m]$, that is, $\alpha_i$ is  the number of occurrences of $i$ in the tableau $T$.

The \emph{reverse} \emph{row} \emph{word} of a semistandard tableau
$T$, denoted $w(T)=w_1\cdots w_l$, with $l$ the number of non zero entries of $T$,  is obtained by reading the entries of its rows (excluding the entry 0) right to left
starting from the top row and proceeding downward. The \emph{weight} of the \emph{word}  $w(T)$ is the weight of $T$. A \emph{Yamanouchi} \emph{word} is a word $w =w_1 \cdots w_l$ such
that, for each $1 \le k \le l$,  the  weight of the subword  $w_1 \cdots w_k$ is a partition.

A semistandard Young tableau $G\in SST_m(\gamma)$ of content $\alpha$ is also realized by the \emph{Gelfand-Tsetlin} (GT) \emph{pattern} of type $\gamma=\gamma^{(m)}$ and weight $\alpha$ \cite{gt50} defined by the sequence of nested partitions
$$\gamma = \gamma^{(m)}\supseteq  \gamma^{(m-1)}\supseteq\cdots\supseteq \gamma^{(1)}$$
where $ \gamma^{(i)}$ of length $\le i$ defines  the filling of the boxes of $G$ on the alphabet $[m]$, for  $1\le i\le m$.
It is
displayed here as a triangular array of
non-negative integers $(\gamma_j^{(i)})_{1\leq j\leq i\leq m}$
as below:
\begin{equation}\label{gtpattern}
\begin{array}{cccccccccccccccccc}
&&&&\gamma_1^{(1)}\cr
&&&\gamma_1^{(2)}&&\gamma_2^{(2)}\cr
&&\cdots&&\cdots&&\cdots\cr
&\gamma_1^{(m-1)}&&\gamma_2^{(m-1)}&&\cdots&&\gamma_{m-1}^{(m-1)}\cr
\gamma_1^{(m)}&&\gamma_2^{(m)}&&\cdots&&\gamma_{m-1}^{(m)}&&\gamma_m^{(m)}&\cr
\end{array}
\end{equation}

The entries satisfy the between-ness conditions $\gamma_j^{(i+1)}\geq \gamma_j^{(i)}\geq \gamma_{j+1}^{(i+1)}$
for $1\leq j\leq i<m$. The $i$th row, enumerated from top to bottom, necessarily constitutes the partition $\gamma^{(i)}$
of length $\leq i$. Its weight is defined to be $\alpha=(|\gamma^{(1)}|-|\gamma^{(0)}|,|\gamma^{(2)}|-|\gamma^{(1)}|,\dots, |\gamma^{(m)}|-|\gamma^{(m-1)}|)$, where  $\gamma^{(0)}=0$.

Semistandard tableaux and Gelfand-Tsetlin patterns are naturally in bijection and we intertwine  the two presentations. Equivalently, $\gamma^{(i)}=G^{-1}([i])$ the pre-image of $[i]$, for $1\le i\le m$. The
semistandard condition translates to the condition that $ \gamma^{(i)}/\gamma^{(i-1 )}$ is a horizontal
strip, that is in any column of the Young diagram of $\gamma^{(i)}$, there is at most one box of
$\gamma^{(i)}$ that is not a box of $\gamma^{(i-1 )}$. In other words, $0\le \ell(\gamma^{(i)})-\ell(\gamma^{(i-1 )})\le 1$.
Henceforth, the first column of $G$ records in  \emph{strictly decreasing order}, from bottom to top, the $i$'s in $[m]$ such that  $\ell(\gamma^{(i)})-1=\ell(\gamma^{(i-1)})$. In other words,
\begin{align}&G(s,1)=b>G(s-1,1)=a, \mbox{ for some $s\in [2,\ell(\gamma)]$}\label{gt1}\\
&\qquad \qquad\Leftrightarrow \nonumber\\
&\ell(\gamma^{(b)})-1=\ell(\gamma^{(a)}), \mbox{ and  $\ell(\gamma^{(x)})=\ell(\gamma^{(a)})$, for all $a\le x<b$}.\label{gt2}
 \end{align}
  Since the shape of $G\in SST_m(\gamma)$ is $\gamma$ and the first column has length $\ell(\gamma)$, the following is an immediate  consequence of the  definition of Gelfand-Tsetlin pattern and  a rephrasing of \eqref{gt1}, \eqref{gt2}.

\begin{lem} \label{lem:leftcomp}Let $n\in\mathbb{N}$. Let $G\in SST_{2n}(\mu)$ with GT pattern of type $\mu$ defined by the nested sequence of partitions
\begin{align}\mu=\mu^{(2n)}\supseteq \mu^{(2n-1)}
\supseteq \cdots \supseteq \mu^{(2)}\supseteq\mu^{(1)}.\label{nestedleft}
\end{align}
The first column of $G$ of length $\ell(\mu)$ read bottom to top is equal to $2n\ge r_{\ell(\mu)}>\cdots>r_i>\cdots >r_1\ge 1$, that is,
\begin{align}G(\ell(\mu),1)=r_{\ell(\mu)}>G(\ell(\mu)-1,1)=r_{\ell(\mu)-1}>\cdots> G(i,1)= r_i >\cdots > G(1,1)=r_1\nonumber\end{align}
if and only if the subsequence of the GT pattern \eqref{nestedleft} consisting of the rows $r_{\ell(\mu)}>\cdots>r_i>\cdots >r_1$, read bottom to top,
 \begin{align}\mu^{(r_{\ell(\mu)})}\supset \mu^{(r_{\ell(\mu)-1})}
\supset \cdots \supset \mu^{(r_i)}\supset\cdots\supset \mu^{(r_2)}\supset\mu^{(r_1)} \label{firstcolumnsequence}
\end{align}
have respectively  lengths  $\ell(\mu), \ell(\mu)-1,\dots, i,\dots,2,1$, with
$\ell(\mu^{(s)})=\ell(\mu^{(r_{i})})$,  for any $r_{i+1}>s\ge  r_{i}$, and  $i=0,1,\dots, \ell(\mu)$.
    We set $\mu^{(r_0)}:=0$, and $r_{\ell{\mu)+1}}:={2n+1}$.

   The subsequence \eqref{firstcolumnsequence} is called the $\emph{1}$-\emph{subarray} of the GT pattern $G$.
 \end{lem}
 Let $G^0=G$.
 For $j=1,\cdots,\mu_1-1$, define $G^j$ to be the GT pattern obtained  by subtracting one unity to  each  positive entry in $G^{j-1}$.

Iterating $\mu_1-1$ times the procedure in Lemma \ref{lem:leftcomp},
the \emph{$j$-subarray} of $G$ is the $1$-subarray of $G^{j-1}$, for $j=1,\dots, \mu_1$. The $j$-subarray of $G$ or $1$-subarray of $G^{j-1}$ encodes the $j$ column of $G\in SST(\mu)$, for $j=1,\dots,\mu_1$.
The GT pattern of the $j$-column of the semistandard tableau $G$ is obtained from the GT pattern $G^{j-1}$ by replacing each positive entry with $1$, for $j=1,\dots,\mu_1$.
 \begin{ex} \label{ex:nested}Let $G\in SST_8(\mu)$ with $\mu=(4,4,2,1)$ and weight $(0,0,1,3,1,3,2,2)$,
 \begin{align}G=  \YT{0.15in}{}{
 {3,4,{4},{6}},
 {4,{5},{7},{7}},
 {6,6,{8}},
 {{8}},
} \mbox{ and its  GT pattern of type $\mu$}
\mbox{$\vcenter{\hbox{\begin{tikzpicture}[x={(1cm*0.2,-1.7320508cm*0.2)},y={(1cm*0.2,1.7320508cm*0.2)}]
\path(0,0)--node[pos=0.45]{$\blue\mathbf{4}$}(1,1);
\path(1,1)--node[pos=0.45]{$\blue\mathbf{4}$}(2,2);
\path(2,2)--node[pos=0.45]{$\blue\mathbf{3}$}(3,3);
\path(3,3)--node[pos=0.45]{$\blue\mathbf{1}$}(4,4);
\path(4,4)--node[pos=0.45]{$0$}(5,5);
\path(5,5)--node[pos=0.45]{$0$}(6,6);
\path(6,6)--node[pos=0.45]{$0$}(7,7);
\path(7,7)--node[pos=0.45]{$0$}(8,8);
\path(0,1)--node[pos=0.45]{$4$}(1,2);
\path(1,2)--node[pos=0.45]{$4$}(2,3);
\path(2,3)--node[pos=0.45]{$2$}(3,4);
\path(3,4)--node[pos=0.45]{$0$}(4,5);
\path(4,5)--node[pos=0.45]{$0$}(5,6);
\path(5,6)--node[pos=0.45]{$0$}(6,7);
\path(6,7)--node[pos=0.45]{$0$}(7,8);
\path(0,2)--node[pos=0.45]{$\blue\mathbf{4}$}(1,3);
\path(1,3)--node[pos=0.45]{$\blue\mathbf{2}$}(2,4);
\path(2,4)--node[pos=0.45]{$\blue\mathbf{2}$}(3,5);
\path(3,5)--node[pos=0.45]{$0$}(4,6);
\path(4,6)--node[pos=0.45]{$0$}(5,7);
\path(5,7)--node[pos=0.45]{$0$}(6,8);
\path(0,3)--node[pos=0.45]{$3$}(1,4);
\path(1,4)--node[pos=0.45]{$2$}(2,5);
\path(2,5)--node[pos=0.45]{$0$}(3,6);
\path(3,6)--node[pos=0.45]{$0$}(4,7);
\path(4,7)--node[pos=0.45]{$0$}(5,8);
\path(0,4)--node[pos=0.45]{$\blue\mathbf{3}$}(1,5);
\path(1,5)--node[pos=0.45]{$\blue\mathbf{1}$}(2,6);
\path(2,6)--node[pos=0.45]{$0$}(3,7);
\path(3,7)--node[pos=0.45]{$0$}(4,8);
\path(0,5)--node[pos=0.45]{$\blue\mathbf{1}$}(1,6);
\path(1,6)--node[pos=0.45]{$0$}(2,7);
\path(2,7)--node[pos=0.45]{$0$}(3,8);
\path(0,6)--node[pos=0.45]{$0$}(1,7);
\path(1,7)--node[pos=0.45]{$0$}(2,8);
\path(0,7)--node[pos=0.45]{$0$}(1,8);
\end{tikzpicture}}}$
}\nonumber
\end{align}
with defining  nested sequence of partitions
\begin{align} 
&\mu=\mu^{(8)}=(4,4,3,1,0^4)\supseteq \mu^{(7)}=(4,4,2,0^4)\supseteq \mu^{(6)}=(4,2,2,0^3)\supseteq\mu^{(5)}=(3,2,0^3)\supseteq \mu^{(4)}=(3,1,0^2)\nonumber \\
&\supseteq \mu^{(3)}=(1,0^2)\supseteq \mu^{(2)}=(0^2)=\mu^{(1)}=(0),\nonumber
\end{align}
and $1$-subarray  $\mu^{(8)}\supseteq \mu^{(6)}\supseteq \mu^{(4)}\supseteq \mu^{(3)}$ in blue.

The first column of $G$ records the supra indices or row numbers $8>6>4>3$ of  the  $1$-subarray 
of the GT pattern $G$,
$$\mu^{(8)}\supseteq \mu^{(6)}\supseteq \mu^{(4)}\supseteq \mu^{(3)},$$
\begin{align}&\ell(\mu^{(8)})=4&>\ell(\mu^{(7)})=\ell(\mu^{(6)})=3&>\ell(\mu^{(5)})=\ell(\mu^{(4)})=2&>\ell(\mu^{(3)})=1&
>\ell(\mu^{(2)})=\ell(\mu^{(1)})=0\nonumber
\end{align}
From the $2$
-subarray of the GT-pattern $G$, the sequence of numbers $6>5>4$ are the entries of the second column of $G$. From the $3$
-subarray of the GT-pattern $G$, the sequence of numbers $8>7>4$ are the entries of the third column of $G$,  and from the $4$-subarray
we get $7>6$ the fourth column of $G$. These $j$-subarrays of $G$, for $j=2,3,4$, are displayed below in blue

\begin{align*}\begin{tikzpicture}[x={(1cm*0.2,-1.7320508cm*0.2)},y={(1cm*0.2,1.7320508cm*0.2)}]
\path(0,0)--node[pos=0.45]{${3}$}(1,1);
\path(1,1)--node[pos=0.45]{${3}$}(2,2);
\path(2,2)--node[pos=0.45]{${2}$}(3,3);
\path(3,3)--node[pos=0.45]{${0}$}(4,4);
\path(4,4)--node[pos=0.45]{$0$}(5,5);
\path(5,5)--node[pos=0.45]{$0$}(6,6);
\path(6,6)--node[pos=0.45]{$0$}(7,7);
\path(7,7)--node[pos=0.45]{$0$}(8,8);
\path(0,1)--node[pos=0.45]{$3$}(1,2);
\path(1,2)--node[pos=0.45]{$3$}(2,3);
\path(2,3)--node[pos=0.45]{$1$}(3,4);
\path(3,4)--node[pos=0.45]{$0$}(4,5);
\path(4,5)--node[pos=0.45]{$0$}(5,6);
\path(5,6)--node[pos=0.45]{$0$}(6,7);
\path(6,7)--node[pos=0.45]{$0$}(7,8);
\path(0,2)--node[pos=0.45]{$\blue\mathbf{3}$}(1,3);
\path(1,3)--node[pos=0.45]{$\blue\mathbf{1}$}(2,4);
\path(2,4)--node[pos=0.45]{$\blue\mathbf{1}$}(3,5);
\path(3,5)--node[pos=0.45]{$0$}(4,6);
\path(4,6)--node[pos=0.45]{$0$}(5,7);
\path(5,7)--node[pos=0.45]{$0$}(6,8);
\path(0,3)--node[pos=0.45]{$\blue\mathbf{2}$}(1,4);
\path(1,4)--node[pos=0.45]{$\blue\mathbf{1}$}(2,5);
\path(2,5)--node[pos=0.45]{$0$}(3,6);
\path(3,6)--node[pos=0.45]{$0$}(4,7);
\path(4,7)--node[pos=0.45]{$0$}(5,8);
\path(0,4)--node[pos=0.45]{$\blue\mathbf{2}$}(1,5);
\path(1,5)--node[pos=0.45]{${0}$}(2,6);
\path(2,6)--node[pos=0.45]{$0$}(3,7);
\path(3,7)--node[pos=0.45]{$0$}(4,8);
\path(0,5)--node[pos=0.45]{${0}$}(1,6);
\path(1,6)--node[pos=0.45]{$0$}(2,7);
\path(2,7)--node[pos=0.45]{$0$}(3,8);
\path(0,6)--node[pos=0.45]{$0$}(1,7);
\path(1,7)--node[pos=0.45]{$0$}(2,8);
\path(0,7)--node[pos=0.45]{$0$}(1,8);
\end{tikzpicture}&\qquad
\begin{tikzpicture}[x={(1cm*0.2,-1.7320508cm*0.2)},y={(1cm*0.2,1.7320508cm*0.2)}]
\path(0,0)--node[pos=0.45]{$\blue\mathbf{2}$}(1,1);
\path(1,1)--node[pos=0.45]{$\blue\mathbf{2}$}(2,2);
\path(2,2)--node[pos=0.45]{$\blue\mathbf{1}$}(3,3);
\path(3,3)--node[pos=0.45]{${0}$}(4,4);
\path(4,4)--node[pos=0.45]{$0$}(5,5);
\path(5,5)--node[pos=0.45]{$0$}(6,6);
\path(6,6)--node[pos=0.45]{$0$}(7,7);
\path(7,7)--node[pos=0.45]{$0$}(8,8);
\path(0,1)--node[pos=0.45]{$\blue\mathbf{2}$}(1,2);
\path(1,2)--node[pos=0.45]{$\blue\mathbf{2}$}(2,3);
\path(2,3)--node[pos=0.45]{$0$}(3,4);
\path(3,4)--node[pos=0.45]{$0$}(4,5);
\path(4,5)--node[pos=0.45]{$0$}(5,6);
\path(5,6)--node[pos=0.45]{$0$}(6,7);
\path(6,7)--node[pos=0.45]{$0$}(7,8);
\path(0,2)--node[pos=0.45]{${2}$}(1,3);
\path(1,3)--node[pos=0.45]{${0}$}(2,4);
\path(2,4)--node[pos=0.45]{${0}$}(3,5);
\path(3,5)--node[pos=0.45]{$0$}(4,6);
\path(4,6)--node[pos=0.45]{$0$}(5,7);
\path(5,7)--node[pos=0.45]{$0$}(6,8);
\path(0,3)--node[pos=0.45]{${1}$}(1,4);
\path(1,4)--node[pos=0.45]{${0}$}(2,5);
\path(2,5)--node[pos=0.45]{$0$}(3,6);
\path(3,6)--node[pos=0.45]{$0$}(4,7);
\path(4,7)--node[pos=0.45]{$0$}(5,8);
\path(0,4)--node[pos=0.45]{$\blue\mathbf{1}$}(1,5);
\path(1,5)--node[pos=0.45]{${0}$}(2,6);
\path(2,6)--node[pos=0.45]{$0$}(3,7);
\path(3,7)--node[pos=0.45]{$0$}(4,8);
\path(0,5)--node[pos=0.45]{${0}$}(1,6);
\path(1,6)--node[pos=0.45]{$0$}(2,7);
\path(2,7)--node[pos=0.45]{$0$}(3,8);
\path(0,6)--node[pos=0.45]{$0$}(1,7);
\path(1,7)--node[pos=0.45]{$0$}(2,8);
\path(0,7)--node[pos=0.45]{$0$}(1,8);
\end{tikzpicture}&
\begin{tikzpicture}[x={(1cm*0.2,-1.7320508cm*0.2)},y={(1cm*0.2,1.7320508cm*0.2)}]
\path(0,0)--node[pos=0.45]{${1}$}(1,1);
\path(1,1)--node[pos=0.45]{${1}$}(2,2);
\path(2,2)--node[pos=0.45]{${0}$}(3,3);
\path(3,3)--node[pos=0.45]{${0}$}(4,4);
\path(4,4)--node[pos=0.45]{$0$}(5,5);
\path(5,5)--node[pos=0.45]{$0$}(6,6);
\path(6,6)--node[pos=0.45]{$0$}(7,7);
\path(7,7)--node[pos=0.45]{$0$}(8,8);
\path(0,1)--node[pos=0.45]{$\blue\mathbf{1}$}(1,2);
\path(1,2)--node[pos=0.45]{$\blue\mathbf{1}$}(2,3);
\path(2,3)--node[pos=0.45]{$0$}(3,4);
\path(3,4)--node[pos=0.45]{$0$}(4,5);
\path(4,5)--node[pos=0.45]{$0$}(5,6);
\path(5,6)--node[pos=0.45]{$0$}(6,7);
\path(6,7)--node[pos=0.45]{$0$}(7,8);
\path(0,2)--node[pos=0.45]{$\blue\mathbf{1}$}(1,3);
\path(1,3)--node[pos=0.45]{${0}$}(2,4);
\path(2,4)--node[pos=0.45]{${0}$}(3,5);
\path(3,5)--node[pos=0.45]{$0$}(4,6);
\path(4,6)--node[pos=0.45]{$0$}(5,7);
\path(5,7)--node[pos=0.45]{$0$}(6,8);
\path(0,3)--node[pos=0.45]{${0}$}(1,4);
\path(1,4)--node[pos=0.45]{${0}$}(2,5);
\path(2,5)--node[pos=0.45]{$0$}(3,6);
\path(3,6)--node[pos=0.45]{$0$}(4,7);
\path(4,7)--node[pos=0.45]{$0$}(5,8);
\path(0,4)--node[pos=0.45]{${0}$}(1,5);
\path(1,5)--node[pos=0.45]{${0}$}(2,6);
\path(2,6)--node[pos=0.45]{$0$}(3,7);
\path(3,7)--node[pos=0.45]{$0$}(4,8);
\path(0,5)--node[pos=0.45]{${0}$}(1,6);
\path(1,6)--node[pos=0.45]{$0$}(2,7);
\path(2,7)--node[pos=0.45]{$0$}(3,8);
\path(0,6)--node[pos=0.45]{$0$}(1,7);
\path(1,7)--node[pos=0.45]{$0$}(2,8);
\path(0,7)--node[pos=0.45]{$0$}(1,8);
\end{tikzpicture}.
\end{align*}
\end{ex}

\section{Symplectic tableaux and symplectic Gelfand-Tsetlin patterns}\label{sec:symplectic}
From now on, we fix $n\in \mathbb{N}$ throughout the paper. Let $\gamma$ be a partition with length $\ell(\gamma)\le 2n$. The partition $\gamma$ is said to be \emph{even} if $\gamma_{2i-1}=\gamma_{2i}$ for all $i\ge 1$. In other words, all columns of $\gamma$ have even length and necessarily the length of $\gamma$ is even.
Let $SST_{2n}(\gamma)$ be the set of all
semistandard tableaux of shape $\gamma$ with entries in $[2n]:=\{1,\dots,2n\}$.

\begin{defi}\cite{king76} \label{def:symp}A semistandard tableau $G \in SST_{2n}(\gamma)$ is said to
be symplectic if
$$G(k, 1) \ge 2k-1, \mbox{  for all $  k \in  [ \ell(\gamma)]$}.$$
Let $SpT_{2n}(\gamma)$ denote the set of all symplectic tableaux of shape $\gamma$ on the alphabet $[2n]$.
\end{defi}

As a consequence of the bijective correspondence between $SST_{2n}(\gamma)$ and the GT pattern presentation one has the definition of symplectic GT pattern.
\begin{defi}\label{def:sympgt} The Gelfand-Tsetlin  pattern $G$ of type $\gamma$ defined by the nested sequence of partitions
\begin{align*}\gamma=\gamma^{(2n)}\supseteq \gamma^{(2n-1)}
\supseteq \cdots \supseteq \gamma^{(2)}\supseteq\gamma^{(1)}.
\end{align*}  is said to be symplectic if its  $1$-subarray
~
\begin{align*}\gamma^{(r_{\ell(\gamma)})}\supset \gamma^{(r_{\ell(\gamma)-1})}
\supset \cdots \supset \gamma^{(r_i)}\supset\cdots\supset \gamma^{(r_1)}
\end{align*}
 satisfy
$$r_k=G(k,1)\ge 2k-1, \mbox{ $k=1,\dots,\ell(\mu).$}$$
\end{defi}

In Example \ref{ex:nested}, $G\in SST_8$  is symplectic because the 1st column $$  \YT{0.15in}{}{
 {3},
 {4},
 {6},
 {{8}},
} \in SpT_8 \mbox{ is symplectic},$$
similarly its GT pattern is symplectic
$$\begin{tikzpicture}[x={(1cm*0.2,-1.7320508cm*0.2)},y={(1cm*0.2,1.7320508cm*0.2)}]
\path(0,0)--node[pos=0.45]{$\blue\mathbf{1}$}(1,1);
\path(1,1)--node[pos=0.45]{$\blue\mathbf{1}$}(2,2);
\path(2,2)--node[pos=0.45]{$\blue\mathbf{1}$}(3,3);
\path(3,3)--node[pos=0.45]{$\blue\mathbf{1}$}(4,4);
\path(4,4)--node[pos=0.45]{$0$}(5,5);
\path(5,5)--node[pos=0.45]{$0$}(6,6);
\path(6,6)--node[pos=0.45]{$0$}(7,7);
\path(7,7)--node[pos=0.45]{$0$}(8,8);
\path(0,1)--node[pos=0.45]{$1$}(1,2);
\path(1,2)--node[pos=0.45]{$1$}(2,3);
\path(2,3)--node[pos=0.45]{$1$}(3,4);
\path(3,4)--node[pos=0.45]{$0$}(4,5);
\path(4,5)--node[pos=0.45]{$0$}(5,6);
\path(5,6)--node[pos=0.45]{$0$}(6,7);
\path(6,7)--node[pos=0.45]{$0$}(7,8);
\path(0,2)--node[pos=0.45]{$\blue\mathbf{1}$}(1,3);
\path(1,3)--node[pos=0.45]{$\blue\mathbf{1}$}(2,4);
\path(2,4)--node[pos=0.45]{$\blue\mathbf{1}$}(3,5);
\path(3,5)--node[pos=0.45]{$0$}(4,6);
\path(4,6)--node[pos=0.45]{$0$}(5,7);
\path(5,7)--node[pos=0.45]{$0$}(6,8);
\path(0,3)--node[pos=0.45]{$1$}(1,4);
\path(1,4)--node[pos=0.45]{$1$}(2,5);
\path(2,5)--node[pos=0.45]{$0$}(3,6);
\path(3,6)--node[pos=0.45]{$0$}(4,7);
\path(4,7)--node[pos=0.45]{$0$}(5,8);
\path(0,4)--node[pos=0.45]{$\blue\mathbf{1}$}(1,5);
\path(1,5)--node[pos=0.45]{$\blue\mathbf{1}$}(2,6);
\path(2,6)--node[pos=0.45]{$0$}(3,7);
\path(3,7)--node[pos=0.45]{$0$}(4,8);
\path(0,5)--node[pos=0.45]{$\blue\mathbf{1}$}(1,6);
\path(1,6)--node[pos=0.45]{$0$}(2,7);
\path(2,7)--node[pos=0.45]{$0$}(3,8);
\path(0,6)--node[pos=0.45]{$0$}(1,7);
\path(1,7)--node[pos=0.45]{$0$}(2,8);
\path(0,7)--node[pos=0.45]{$0$}(1,8);
\end{tikzpicture}.$$

The following proposition, due to Watanabe \cite{watanabe}, characterizes the minimal row number of $G\notin SpT_{2n}(\gamma)$ where a symplectic violation occurs.
Our  main result Theorem \ref{th:main} chracterizes the maximal row number of $G$, as a left companion of an LR  tableau, where a symplectic violation occurs. See Examples \ref{ex:example}.
\begin{prop} \cite{watanabe} 
Let $G \in SST_{2n}(\gamma)$.
\begin{enumerate}
\item If $G\in SpT_{2n}(\gamma)$ then $\ell(\gamma)\le n$.
\item If $G$ is not symplectic, then there exists a unique
$i \in[2, 2n]$ such that
\begin{align}G(i, 1) < 2i -1 \mbox{ and } G(k, 1) \ge 2k -1  \mbox{ for all $k \in [1, i -1 ]$}.\end{align}
Moreover, we have
\begin{align}G(i - 1, 1) = 2i - 3=2(i-1)-1  \mbox{  and $G(i, 1) = 2i -2=2(i-1)$ }.\end{align}
\end{enumerate}
\end{prop}

The GT pattern
$$\begin{tikzpicture}[x={(1cm*0.2,-1.7320508cm*0.2)},y={(1cm*0.2,1.7320508cm*0.2)}]
\path(0,0)--node[pos=0.45]{$\blue\mathbf{4}$}(1,1);
\path(1,1)--node[pos=0.45]{$\blue\mathbf{4}$}(2,2);
\path(2,2)--node[pos=0.45]{$\blue\mathbf{3}$}(3,3);
\path(3,3)--node[pos=0.45]{$\blue\mathbf{1}$}(4,4);
\path(4,4)--node[pos=0.45]{$0$}(5,5);
\path(5,5)--node[pos=0.45]{$0$}(6,6);
\path(0,1)--node[pos=0.45]{$4$}(1,2);
\path(1,2)--node[pos=0.45]{$4$}(2,3);
\path(2,3)--node[pos=0.45]{$2$}(3,4);
\path(3,4)--node[pos=0.45]{$0$}(4,5);
\path(4,5)--node[pos=0.45]{$0$}(5,6);
\path(0,2)--node[pos=0.45]{$\blue\mathbf{4}$}(1,3);
\path(1,3)--node[pos=0.45]{$\blue\mathbf{2}$}(2,4);
\path(2,4)--node[pos=0.45]{$\blue\mathbf{2}$}(3,5);
\path(3,5)--node[pos=0.45]{$0$}(4,6);
\path(0,3)--node[pos=0.45]{$3$}(1,4);
\path(1,4)--node[pos=0.45]{$2$}(2,5);
\path(2,5)--node[pos=0.45]{$0$}(3,6);
\path(0,4)--node[pos=0.45]{$\blue\mathbf{3}$}(1,5);
\path(1,5)--node[pos=0.45]{$\blue\mathbf{1}$}(2,6);
\path(0,5)--node[pos=0.45]{$\blue\mathbf{2}$}(1,6);
\end{tikzpicture}$$
with $1$-subarray $\gamma^{(6)}\supset \gamma^{(4)}
\supset  \gamma^{(2)}\supset\gamma^{(1)}$ is not symplectic because the column
$$  \YT{0.15in}{}{
 {1},
 {2},
 {4},
 {{6}},
} \in SSYT_6 \mbox{ is not symplectic.}$$ On the other hand $\ell(\gamma)>3$.
Every GT pattern can be transformed into a symplectic one by adding enough downward diagonals of zeroes on the right has in Example \ref{ex:nested}.

\section{Symplectic Littlewood-Richardson tableaux and the symplectic left companions}\label{sec:lrsymplectic}
Let $\lambda$ be a partition with $\ell(\lambda)\le 2n$.
Let $\mu,\nu \subseteq \lambda$ with $\ell(\mu)\le n$ and $\nu$ an even partition. 
\subsection{LR tableaux of even weight}
A tableau $T \in  SST_{2n}(\lambda/\mu)$ of weight $\nu$ is said to be
 Littlewood-Richardson (LR) tableau if its reverse row word is a Yamanouchi word of weight $\nu$.
Let $LR(\lambda/\mu,\nu)$ be the set of all LR tableaux of shape $\lambda/\mu$ and weight $\nu$ on the alphabet $[2n]$.

\begin{obs}Note $2n\ge \ell(\lambda)\Leftrightarrow \ell(\lambda)-n\le n$.
\end{obs}
\begin{lem} Let $T\in LR(\lambda/\mu,\nu)$. Let $i\ge 0$.
\begin{enumerate}
\item If $T(k,j)=2i+1$ then $T(k',j')=2(i+1)$ for some $k'>k$ and $j'$.
\item If $T(k,1)=2i+1$ then $T(k+1,1)=2(i+1)$.
\item  For any $1\le k,j$ $T(k,j)\le k$.
\end{enumerate}
\end{lem}

\begin{proof}By assumption $\nu$ is an even partition and the reverse row word  of $T$ is a Yamanouchi word.
\end{proof}

\subsection{The left companion of an LR tableau}
In \cite{sathishtorres} the right companion of a Littlewood-Richardson-Sundaram tableau has been characterized by a flag condition.
Next section characterizes  the left companion by the symplectic property. We  first recall the definition of   \emph{left} \emph{Gelfand}-\emph{Tsetlin} \emph{pattern}, or  \emph{left} \emph{companion} \emph{tableu},  for short \emph{left companion}, of an LR tableau.

\begin{defi}\cite{GZ85,gzpolyed, BZphy}, \cite{akt16}\label{leftcomp} Let $T\in LR(\lambda/\mu,\nu)$. The left companion of $T$, $G_\mu(T)\in SST_{2n}(\mu)$ of  shape $\mu$ in the alphabet $[2n]$ and content $rev(\lambda-\nu)$ the reverse of $\lambda-\nu$, is  obtained from $T$ by recording the sequence of partitions $\mu^{(2n-r+1)}$ giving the
shapes occupied by the entries $<r$, including the empty entries of the shape $\mu$  identified with $0$,  in rows $r, r + 1, \dots , 2n$ of $T$, for $r = 1, 2, \dots, 2n$. We  then get the nested sequence of partitions $\mu=\mu^{(2n)}\supseteq \mu^{(2n-1)}\supseteq \cdots \supseteq\mu^{(1)}$ defining $G_\mu(T)$ as the left Gelfand-Tsetlin pattern.
\end{defi}
It follows from the definition of the left companion $G_\mu(T)$ that the $i$'s in row $i$ of $T$ do not contribute for  its construction, for $i=1,\dots,2n$. Indeed, adding the column partition  $(1^k)$, with $k$ even  to $T$  will change  $\lambda$ to $\lambda+(1^k)$ and $\nu$ to $\nu+(1^k)$ still an even partition and the weight of $G_\mu(T)$ is preserved. This is illustrated below.

\begin{ex}\label{ex:companiontableaux}
\begin{enumerate}
\item Let $n=3$, $\lambda=(4,3,2,1,1,0)$, $\nu=(1^5,0)$ and $\lambda'=\mu=(3,2,1,0^3)$ and $\nu'=(0)$,
\begin{align} T=\YT{0.15in}{}{
 {{},{},,1},
 {{},,2},
 {{},3},
 {{4}},
{5},
}, \quad  T'=\YT{0.15in}{}{
 {{},{},},
 {{},},
 {{}},
{},
},\quad G_\mu(T)=G_\mu(T')=\YT{0.15in}{}{
 {4,5,6},
 {{5},{6}},
{6},
}
\quad
\begin{tikzpicture}[x={(1cm*0.2,-1.7320508cm*0.2)},y={(1cm*0.2,1.7320508cm*0.2)}]
\path(0,0)--node[pos=0.45]{$3$}(1,1);
\path(1,1)--node[pos=0.45]{$2$}(2,2);
\path(2,2)--node[pos=0.45]{$1$}(3,3);
\path(3,3)--node[pos=0.45]{$0$}(4,4);
\path(4,4)--node[pos=0.45]{$0$}(5,5);
\path(5,5)--node[pos=0.45]{$0$}(6,6);
\path(0,1)--node[pos=0.45]{$2$}(1,2);
\path(1,2)--node[pos=0.45]{$1$}(2,3);
\path(2,3)--node[pos=0.45]{$0$}(3,4);
\path(3,4)--node[pos=0.45]{$0$}(4,5);
\path(4,5)--node[pos=0.45]{$0$}(5,6);
\path(0,2)--node[pos=0.45]{$1$}(1,3);
\path(1,3)--node[pos=0.45]{$0$}(2,4);
\path(2,4)--node[pos=0.45]{$0$}(3,5);
\path(3,5)--node[pos=0.45]{$0$}(4,6);
\path(0,3)--node[pos=0.45]{$0$}(1,4);
\path(1,4)--node[pos=0.45]{$0$}(2,5);
\path(2,5)--node[pos=0.45]{$0$}(3,6);
\path(0,4)--node[pos=0.45]{$0$}(1,5);
\path(1,5)--node[pos=0.45]{$0$}(2,6);
\path(0,5)--node[pos=0.45]{$0$}(1,6);
\end{tikzpicture}
\end{align}

\item Let $n=2$, $\lambda=(2,1^3),$ and $\lambda'=(3,1^3)$, $\mu=(1^2,0^2)$, $\nu=(1,1,0^2)$, $\nu'=(2,1,0^2)$ and $T\in LR(\lambda/\mu,\nu)$ and
$T'\in LR(\lambda'/\mu,\nu')$. One has $\lambda-\nu=\lambda'-\nu'=(1,1,0,1)$, and
\begin{align} T=\YT{0.15in}{}{
 {{},{}},
 {{}},
 {{1}},
{2},
},\quad G_\mu(T)=\YT{0.15in}{}{
 {{1},{4}},
{2},
}
\quad
\begin{tikzpicture}[x={(1cm*0.2,-1.7320508cm*0.2)},y={(1cm*0.2,1.7320508cm*0.2)}]
\path(0,0)--node[pos=0.45]{$2$}(1,1);
\path(1,1)--node[pos=0.45]{$1$}(2,2);
\path(2,2)--node[pos=0.45]{$0$}(3,3);
\path(3,3)--node[pos=0.45]{$0$}(4,4);
\path(0,1)--node[pos=0.45]{$1$}(1,2);
\path(1,2)--node[pos=0.45]{$1$}(2,3);
\path(2,3)--node[pos=0.45]{$0$}(3,4);
\path(0,2)--node[pos=0.45]{$1$}(1,3);
\path(1,3)--node[pos=0.45]{$1$}(2,4);
\path(0,3)--node[pos=0.45]{$1$}(1,4);
\end{tikzpicture};\quad
 T'=\YT{0.15in}{}{
 {{},{},1},
 {{},2},
 {{1},3},
{2,4},
},\quad G_\mu(T')=\YT{0.15in}{}{
 {{1},{4}},
{2},
}
\quad
\begin{tikzpicture}[x={(1cm*0.2,-1.7320508cm*0.2)},y={(1cm*0.2,1.7320508cm*0.2)}]
\path(0,0)--node[pos=0.45]{$2$}(1,1);
\path(1,1)--node[pos=0.45]{$1$}(2,2);
\path(2,2)--node[pos=0.45]{$0$}(3,3);
\path(3,3)--node[pos=0.45]{$0$}(4,4);
\path(0,1)--node[pos=0.45]{$1$}(1,2);
\path(1,2)--node[pos=0.45]{$1$}(2,3);
\path(2,3)--node[pos=0.45]{$0$}(3,4);
\path(0,2)--node[pos=0.45]{$1$}(1,3);
\path(1,3)--node[pos=0.45]{$1$}(2,4);
\path(0,3)--node[pos=0.45]{$1$}(1,4);
\end{tikzpicture}
\end{align}

\item Let $n=4$, $\lambda=(6,5,5,4,3,1,0^2)$ even, $\nu=(4,4,3,3,0^4) $, $\mu=(4,3,2,1,0^4)$  and $T\in  LR(\lambda/\mu,\nu)$ as below. We illustrate $T$ with its left companion $G_\mu(T)=G$
\begin{align} &T=\YT{0.15in}{}{
 {{},{},{},{},{1},{1}},
 {{},{},{},{1},{2}},
 {{},{},{1},{2},{3}},
 {{},{2},{3},{4}},
 {{2},{3},{4}},
{4}},&\quad
G_\mu(T)=  \YT{0.15in}{}{
 {3,4,{4},{6}},
 {4,{5},{7}},
 {6,{8}},
 {{8}},
}&
\quad
\begin{tikzpicture}[x={(1cm*0.2,-1.7320508cm*0.2)},y={(1cm*0.2,1.7320508cm*0.2)}]
\path(0,0)--node[pos=0.45]{$4$}(1,1);
\path(1,1)--node[pos=0.45]{$3$}(2,2);
\path(2,2)--node[pos=0.45]{$2$}(3,3);
\path(3,3)--node[pos=0.45]{$1$}(4,4);
\path(4,4)--node[pos=0.45]{$0$}(5,5);
\path(5,5)--node[pos=0.45]{$0$}(6,6);
\path(6,6)--node[pos=0.45]{$0$}(7,7);
\path(7,7)--node[pos=0.45]{$0$}(8,8);
\path(0,1)--node[pos=0.45]{$4$}(1,2);
\path(1,2)--node[pos=0.45]{$3$}(2,3);
\path(2,3)--node[pos=0.45]{$1$}(3,4);
\path(3,4)--node[pos=0.45]{$0$}(4,5);
\path(4,5)--node[pos=0.45]{$0$}(5,6);
\path(5,6)--node[pos=0.45]{$0$}(6,7);
\path(6,7)--node[pos=0.45]{$0$}(7,8);
\path(0,2)--node[pos=0.45]{$4$}(1,3);
\path(1,3)--node[pos=0.45]{$2$}(2,4);
\path(2,4)--node[pos=0.45]{$1$}(3,5);
\path(3,5)--node[pos=0.45]{$0$}(4,6);
\path(4,6)--node[pos=0.45]{$0$}(5,7);
\path(5,7)--node[pos=0.45]{$0$}(6,8);
\path(0,3)--node[pos=0.45]{$3$}(1,4);
\path(1,4)--node[pos=0.45]{$2$}(2,5);
\path(2,5)--node[pos=0.45]{$0$}(3,6);
\path(3,6)--node[pos=0.45]{$0$}(4,7);
\path(4,7)--node[pos=0.45]{$0$}(5,8);
\path(0,4)--node[pos=0.45]{$3$}(1,5);
\path(1,5)--node[pos=0.45]{$1$}(2,6);
\path(2,6)--node[pos=0.45]{$0$}(3,7);
\path(3,7)--node[pos=0.45]{$0$}(4,8);
\path(0,5)--node[pos=0.45]{$1$}(1,6);
\path(1,6)--node[pos=0.45]{$0$}(2,7);
\path(2,7)--node[pos=0.45]{$0$}(3,8);
\path(0,6)--node[pos=0.45]{$0$}(1,7);
\path(1,7)--node[pos=0.45]{$0$}(2,8);
\path(0,7)--node[pos=0.45]{$0$}(1,8);
\end{tikzpicture}
\nonumber
\end{align}
$G_\mu(T)$ has weight $rev(\lambda-\nu)=(0^2,1,3,1,2,1,2)$ and as the left Gelfand-Tsetlin pattern of $T$ is defined by the nested sequence of partitions constructed as in Definition \ref{leftcomp},

\begin{align}
&\mu=\mu^{(8)}=(4,3,2,1,0^4)\supseteq \mu^{(8-1)}=(4,3,1,0^4)\supseteq \mu^{(8-2)}=(4,2,1,0^3)\supseteq\mu^{(8-3)}=(3,2,0^3)\nonumber \\
&\supseteq \mu^{(8-4)}=(3,1,0^2)
\supseteq \mu^{(8-5)}=(1,0^2)\supseteq \mu^{(8-6)}=(0^2)=\mu^{(8-7)}=(0).\label{mainnested}
\end{align}
\noindent which gives  the fourth column  of $G_\mu(T)$.
\end{enumerate}
\end{ex}

As the $i$'s in the row $i$ of an LR tableau $T$ do not contribute for  the construction of $G_\mu(T)$, every GT pattern of type $\mu$ and weight $\alpha$ specify, as left companion,  an  LR tableau $T\in LR(\lambda/\mu,\nu)$ where $rev(\lambda-\nu)=\alpha$ where $rev$ means reverse. More precisely, from   the Berenstein-Zelevinky LR rule, Theorem ~4.3 in~\cite{BZphy} (see also \cite{akt16}), the set of left companions of $LR(\lambda/\mu,\nu)$, denoted
 ${}^- LR^\lambda_{\mu,\nu}$,  are those tableaux or GT patterns $G$ of type $\mu$ and content $ rev(\lambda/\nu)$, where the defining nested sequence of partitions $\mu=\mu^{(2n)}\supseteq \mu^{(2n-1)}\supseteq \cdots \supseteq\mu^{(1)}$
  content $ rev(\lambda/\nu)$  satisfies certain linear inequalities pertaining the partition $\nu$. (See also \cite[Section 2.4.2]{az18v5} for a crystal theoretical characterization.)

\begin{ex}\label{ex:sympviolation}
Let $n=5$, and \begin{align*}G= \YT{0.15in}{}{
 {1},
 {\red\mathbf{2}},
 {\red\mathbf{3}},
 {\red\mathbf{4}},
 {\red\mathbf{8}},
}\in SST_{10}(1^5,0^5)\setminus SpT_{10}(1^5,0^5)\end{align*} (symplectic failing in red),  with GT pattern of type $\mu=(1^5,0^5)$ and weight $\alpha=(1^4,0^3,1,0^2)$ depicted below on the left hand side of \eqref{n=five},
\begin{align}&\label{n=five} \begin{tikzpicture}[x={(1cm*0.2,-1.7320508cm*0.2)},y={(1cm*0.2,1.7320508cm*0.2)}]
\path(0,0)--node[pos=0.45]{$1$}(1,1);
\path(1,1)--node[pos=0.45]{$1$}(2,2);
\path(2,2)--node[pos=0.45]{$1$}(3,3);
\path(3,3)--node[pos=0.45]{$1$}(4,4);
\path(4,4)--node[pos=0.45]{$1$}(5,5);
\path(5,5)--node[pos=0.45]{$0$}(6,6);
\path(6,6)--node[pos=0.45]{$0$}(7,7);
\path(7,7)--node[pos=0.45]{$0$}(8,8);
\path(8,8)--node[pos=0.45]{$0$}(9,9);
\path(9,9)--node[pos=0.45]{$0$}(10,10);
\path(0,1)--node[pos=0.45]{$1$}(1,2);
\path(1,2)--node[pos=0.45]{$1$}(2,3);
\path(2,3)--node[pos=0.45]{$1$}(3,4);
\path(3,4)--node[pos=0.45]{$1$}(4,5);
\path(4,5)--node[pos=0.45]{$1$}(5,6);
\path(5,6)--node[pos=0.45]{$0$}(6,7);
\path(6,7)--node[pos=0.45]{$0$}(7,8);
\path(7,8)--node[pos=0.45]{$0$}(8,9);
\path(8,9)--node[pos=0.45]{$0$}(9,10);
\path(0,2)--node[pos=0.45]{$\blue\mathbf{1}$}(1,3);
\path(1,3)--node[pos=0.45]{$\blue\mathbf{1}$}(2,4);
\path(2,4)--node[pos=0.45]{$\blue\mathbf{1}$}(3,5);
\path(3,5)--node[pos=0.45]{$\blue\mathbf{1}$}(4,6);
\path(4,6)--node[pos=0.45]{$\blue\mathbf{1}$}(5,7);
\path(5,7)--node[pos=0.45]{$0$}(6,8);
\path(6,8)--node[pos=0.45]{$0$}(7,9);
\path(7,9)--node[pos=0.45]{$0$}(8,10);
\path(0,3)--node[pos=0.45]{$1$}(1,4);
\path(1,4)--node[pos=0.45]{$1$}(2,5);
\path(2,5)--node[pos=0.45]{$1$}(3,6);
\path(3,6)--node[pos=0.45]{$1$}(4,7);
\path(4,7)--node[pos=0.45]{$0$}(5,8);
\path(5,8)--node[pos=0.45]{$0$}(6,9);
\path(6,9)--node[pos=0.45]{$0$}(7,10);
\path(0,4)--node[pos=0.45]{$1$}(1,5);
\path(1,5)--node[pos=0.45]{$1$}(2,6);
\path(2,6)--node[pos=0.45]{$1$}(3,7);
\path(3,7)--node[pos=0.45]{$1$}(4,8);
\path(4,8)--node[pos=0.45]{$0$}(5,9);
\path(5,9)--node[pos=0.45]{$0$}(6,10);
\path(0,5)--node[pos=0.45]{$1$}(1,6);
\path(1,6)--node[pos=0.45]{$1$}(2,7);
\path(2,7)--node[pos=0.45]{$1$}(3,8);
\path(3,8)--node[pos=0.45]{$1$}(4,9);
\path(4,9)--node[pos=0.45]{$0$}(5,10);
\path(0,6)--node[pos=0.45]{$\blue\mathbf{1}$}(1,7);
\path(1,7)--node[pos=0.45]{$\blue\mathbf{1}$}(2,8);
\path(2,8)--node[pos=0.45]{$\blue\mathbf{1}$}(3,9);
\path(3,9)--node[pos=0.45]{$\blue\mathbf{1}$}(4,10);
\path(0,7)--node[pos=0.45]{$\blue\mathbf{1}$}(1,8);
\path(1,8)--node[pos=0.45]{$\blue\mathbf{1}$}(2,9);
\path(2,9)--node[pos=0.45]{$\blue\mathbf{1}$}(3,10);
\path(0,8)--node[pos=0.45]{$\blue\mathbf{1}$}(1,9);
\path(1,9)--node[pos=0.45]{$\blue\mathbf{1}$}(2,10);
\path(0,9)--node[pos=0.45]{$\blue\mathbf{1}$}(1,10);
\end{tikzpicture}\quad\quad \begin{tikzpicture}[x={(1cm*0.2,-1.7320508cm*0.2)},y={(1cm*0.2,1.7320508cm*0.2)}]
\path(0,0)--node[pos=0.45]{$1$}(1,1);
\path(1,1)--node[pos=0.45]{$1$}(2,2);
\path(2,2)--node[pos=0.45]{$1$}(3,3);
\path(3,3)--node[pos=0.45]{$1$}(4,4);
\path(4,4)--node[pos=0.45]{$1$}(5,5);
\path(5,5)--node[pos=0.45]{$0$}(6,6);
\path(6,6)--node[pos=0.45]{$0$}(7,7);
\path(7,7)--node[pos=0.45]{$0$}(8,8);
\path(8,8)--node[pos=0.45]{$0$}(9,9);
\path(9,9)--node[pos=0.45]{$0$}(10,10);
\path(10,10)--node[pos=0.45]{$0$}(11,11);
\path(11,11)--node[pos=0.45]{$0$}(12,12);
\path(0,1)--node[pos=0.45]{$1$}(1,2);
\path(1,2)--node[pos=0.45]{$1$}(2,3);
\path(2,3)--node[pos=0.45]{$1$}(3,4);
\path(3,4)--node[pos=0.45]{$1$}(4,5);
\path(4,5)--node[pos=0.45]{$1$}(5,6);
\path(5,6)--node[pos=0.45]{$0$}(6,7);
\path(6,7)--node[pos=0.45]{$0$}(7,8);
\path(7,8)--node[pos=0.45]{$0$}(8,9);
\path(8,9)--node[pos=0.45]{$0$}(9,10);
\path(9,10)--node[pos=0.45]{$0$}(10,11);
\path(10,11)--node[pos=0.45]{$0$}(11,12);
\path(0,2)--node[pos=0.45]{$\blue\mathbf{1}$}(1,3);
\path(1,3)--node[pos=0.45]{$\blue\mathbf{1}$}(2,4);
\path(2,4)--node[pos=0.45]{$\blue\mathbf{1}$}(3,5);
\path(3,5)--node[pos=0.45]{$\blue\mathbf{1}$}(4,6);
\path(4,6)--node[pos=0.45]{$\blue\mathbf{1}$}(5,7);
\path(5,7)--node[pos=0.45]{$0$}(6,8);
\path(6,8)--node[pos=0.45]{$0$}(7,9);
\path(7,9)--node[pos=0.45]{$0$}(8,10);
\path(8,10)--node[pos=0.45]{$0$}(9,11);
\path(9,11)--node[pos=0.45]{$0$}(10,12);
\path(0,3)--node[pos=0.45]{$1$}(1,4);
\path(1,4)--node[pos=0.45]{$1$}(2,5);
\path(2,5)--node[pos=0.45]{$1$}(3,6);
\path(3,6)--node[pos=0.45]{$1$}(4,7);
\path(4,7)--node[pos=0.45]{$0$}(5,8);
\path(5,8)--node[pos=0.45]{$0$}(6,9);
\path(6,9)--node[pos=0.45]{$0$}(7,10);
\path(7,10)--node[pos=0.45]{$0$}(8,11);
\path(8,11)--node[pos=0.45]{$0$}(9,12);
\path(0,4)--node[pos=0.45]{$1$}(1,5);
\path(1,5)--node[pos=0.45]{$1$}(2,6);
\path(2,6)--node[pos=0.45]{$1$}(3,7);
\path(3,7)--node[pos=0.45]{$1$}(4,8);
\path(4,8)--node[pos=0.45]{$0$}(5,9);
\path(5,9)--node[pos=0.45]{$0$}(6,10);
\path(6,10)--node[pos=0.45]{$0$}(7,11);
\path(7,11)--node[pos=0.45]{$0$}(8,12);
\path(0,5)--node[pos=0.45]{$1$}(1,6);
\path(1,6)--node[pos=0.45]{$1$}(2,7);
\path(2,7)--node[pos=0.45]{$1$}(3,8);
\path(3,8)--node[pos=0.45]{$1$}(4,9);
\path(4,9)--node[pos=0.45]{$0$}(5,10);
\path(5,10)--node[pos=0.45]{${0}$}(6,11);
\path(6,11)--node[pos=0.45]{${0}$}(7,12);
\path(0,6)--node[pos=0.45]{$\blue\mathbf{1}$}(1,7);
\path(1,7)--node[pos=0.45]{$\blue\mathbf{1}$}(2,8);
\path(2,8)--node[pos=0.45]{$\blue\mathbf{1}$}(3,9);
\path(3,9)--node[pos=0.45]{$\blue\mathbf{1}$}(4,10);
\path(4,10)--node[pos=0.45]{${0}$}(5,11);
\path(5,11)--node[pos=0.45]{${0}$}(6,12);
\path(0,7)--node[pos=0.45]{$\blue\mathbf{1}$}(1,8);
\path(1,8)--node[pos=0.45]{$\blue\mathbf{1}$}(2,9);
\path(2,9)--node[pos=0.45]{$\blue\mathbf{1}$}(3,10);
\path(3,10)--node[pos=0.45]{${0}$}(4,11);
\path(4,11)--node[pos=0.45]{${0}$}(5,12);
\path(0,8)--node[pos=0.45]{$\blue\mathbf{1}$}(1,9);
\path(1,9)--node[pos=0.45]{$\blue\mathbf{1}$}(2,10);
\path(2,10)--node[pos=0.45]{${0}$}(3,11);
\path(3,11)--node[pos=0.45]{${0}$}(4,12);
\path(0,9)--node[pos=0.45]{$\blue\mathbf{1}$}(1,10);
\path(1,10)--node[pos=0.45]{${0}$}(2,11);
\path(2,11)--node[pos=0.45]{${0}$}(3,12);
\path(0,10)--node[pos=0.45]{${0}$}(1,11);
\path(1,11)--node[pos=0.45]{${0}$}(2,12);
\path(0,11)--node[pos=0.45]{${0}$}(1,12);
\end{tikzpicture}
\end{align}
Let $n=6$, and  add two SE downward diagonals of zeroes  to  the GT pattern on the left hand side of \eqref{n=five} to get the right hand one, then $G$ becomes (not yet symplectic), \begin{align*}G= \YT{0.15in}{}{
 {3},
 {{4}},
 {5},
 {\red\mathbf{6}},
 {{10}},
}\in SST_{12}(1^5,0^7)\setminus SpT_{12}(1^5,0^7) \mbox{ the left companion of $T$ for $n=6$.}
\end{align*}

 Let $n=7$, and  add two more SE downward diagonals of zeroes  to the GT pattern on right hand side  \eqref{n=five}, then $G$ becomes symplectic \begin{align*}G= \YT{0.15in}{}{
 {5},
 {{6}},
 {7},
 {{8}},
 {{12}},
}\in  SpT_{14}(1^5,0^9) \mbox{ the left companion tableau of $T$ for $n=7$.}
\end{align*}

They  specify,  as left companions, respectively for $n=5, 6,7$,
\begin{align}
 T=\YT{0.15in}{}{
 {{},{1}},
 {{},{2}},
 {{},{3}},
 {{}},
  {{}},
 {{1}},
{{2}},
{{4}},
{{5}},
{{6}},
} \in LR(\lambda/\mu,\nu),\; \lambda=(2^3,1^7,0^{2n-10}),\; \nu=(2,2,1^4,0^{2n-6}).
\end{align}
\end{ex}

The next proposition collects a few properties of $G_\mu(T)$ as a left companion of $T$.

\begin{prop} Let $T\in LR(\lambda/\mu,\nu)$ and the left companion of $G_\mu(T)$ defined by \eqref{nestedleft}. The following holds
\begin{enumerate}
\item $T(k,j)$ is not defined or  $T(k,j)\le k$, for any $1\le k,j$. In particular, $T(\ell(\mu)+s,1)\in [1,\ell(\mu)+s]$, for $s\ge 1$.
\item   $T(\ell(\mu)+s,1)=\ell(\mu)+s$, for some  $s\ge 1$, or not defined if and only if
\begin{align}\label{zero}\mu^{(2n-(\ell(\mu)+t)+1)}=(0) \mbox{ for all $t\ge s$}.
\end{align}
In this case, $G_\mu(T)=G_\mu(T')$ where $T'$ is obtained from $T$ by erasing the  $i$'s in each row $i= \ell(\mu)+t$ for $t\ge s$.

\item Either $T(k,1)$  is not defined for $k\ge 1$, in which case $\ell(\lambda)=\ell(\mu)$, or $T(k,1)= k$ for $\ell (\mu)<k\le \ell (\lambda)\le 2n$ if and only if
\begin{align}\label{prop:sequencestrictlength}\ell(\mu^{(2n)})>\ell(\mu^{(2n-1)})>\cdots>
  \ell(\mu^{(2n-\ell(\mu)+1)})>\ell(\mu^{(2n-\ell(\mu))})=\cdots=\ell(\mu^{(1)})=0.
   \end{align}
   In this case, $G_\mu(T)=G_\mu(T')$ where $T'$ is obtained from $T$ by erasing the  $i$'s in each row $ \ell(\mu)< i\le \ell(\lambda)$.
   \item $T(k,\mu_k+1)= k$ for $1\le k\le \ell(\lambda)$ if and only if
   $G_\mu(T)=G_\mu(D(\mu))$ is  the evacuation  \cite{stanley,fulton} of the Yamanouchi tableau of shape $\mu$ on the alphabet $[2n]$. In particular, the first column is $2n>2n-1>\cdots,2n-\ell(\mu)+1$,
\begin{align}\begin{array}{cccccccccccccccccc}
&&&&&&&0\cr
&&&&&&\cdots&&\cdots&&\cr
&&&&&0&&\cdots&&0\cr
&&&&\mu_{\ell(\mu)}&&0&&\cdots&&0\cr
&&&\mu_{\ell(\mu)-1}&&\mu_{\ell(\mu)}&&0&&\cdots&&0\cr
&&\cdots&&\cdots&&\cdots&&0&&\cdots&&0\cr
&\mu_2&&\mu_3&&\cdots&&\mu_{\ell(\mu)}&&0&&\cdots&&0\cr
\mu_1&&\mu_2&&\cdots&&\mu_{\ell(\mu)-1}&&\mu_{\ell(\mu)}&&0&&\cdots&&0\cr
\end{array}
\end{align}

   \item Let $T(\ell(\mu)+1,1)=s\in [\ell(\mu)]$. Then either $T(\ell(\mu)+1,1)= s$ even, or  $s$ odd and for some even $2\le t\le \ell(\mu)-s+2$, one has $T(\ell(\mu)+i,1)=s+i-1$,  $ 1\le i\le t$. Moreover, this is equivalent to

       \begin{align}&\ell(\mu^{(2n)})>\ell(\mu^{(2n-1)})>\cdots>
  \ell(\mu^{(2n-s+1)})=\ell(\mu^{(2n-s)}), \mbox{ if  $s$ even},\\
  &\mbox{ and,  for some even $2\le t\le \ell(\mu)+1$, } \nonumber\\
  &  \ell(\mu^{(2n)})>\ell(\mu^{(2n-1)})>\cdots>
  \ell(\mu^{(2n-s+1)})=\ell(\mu^{(2n-s)})=\cdots=\ell(\mu^{(2n-(s+t-1))}), \mbox{ if  $s$ odd}.
  \end{align}

\end{enumerate}
\end{prop}
\begin{proof}$(1)$
Recall since $T$ is LR,  above or in a row $k$ of  $T$ there are no larger entries  than $k$, that is,  $T(k,j)$ is not defined or  $T(k,j)\le k$, for any $1\le k,j$. In particular, $T(k,1)\le k$ or not defined for $k>\ell(\mu)$.

$(2)$ The partition $\mu^{(2n-(\ell(\mu)+t)+1)}$ gives the
shape occupied by the entries $<\ell(\mu)+t$, including the empty entries identified with $0$,  in rows $\ell(\mu)+t, \ell(\mu)+t + 1, \dots , 2n$ of $T$. Since $T$ is LR and  $T(\ell(\mu)+s,1)=\ell(\mu)+s< T(\ell(\mu)+t,1)$, it forces
$T(\ell(\mu)+t,1)=\ell(\mu)+t\le T(\ell(\mu)+t,j) $, for $t\ge s$ and $j\ge 1$. Therefore, there are no entries $<\ell(\mu)+t$ in rows $\ell(\mu)+t, \ell(\mu)+t + 1, \dots , 2n$ of $T$, and $\ell(\mu^{(2n-(\ell(\mu)+t)+1)})=0$  for all $t\ge s$. In conclusion the $i$'s in row $i$ of $T$ do not contribute for $G_\mu(T)$.

$(3)$ It is a consequence of $(2)$. For $r= \ell(\mu)+1,\dots, 2n$,
the  partition $\mu^{(2n-r+1)}$ giving the
shape occupied by the entries $<r$, including the empty entries of the shape $\mu$  identified with $0$,  in rows $r, r + 1, \dots , 2n$ of $T$, is empty.

For $1\le r\le \ell(\mu)$, the  partition $\mu^{(2n-r+1)}$ giving the
shape occupied by the entries $<r$, including the empty entries of the shape $\mu$  identified with $0$,  in rows $r, r + 1, \dots , 2n$ of $T$, has length $\ell(\mu)-r+1$.

We  then get the nested sequence of partitions defining $G_\mu(T)$ to be

$$\mu=\mu^{(2n)}\supset \mu^{(2n-1)}\supset \cdots \supset\mu^{(2n-\ell(\mu)+1)}\supset \mu^{(2n-\ell(\mu))}=\cdots=\mu^{(1)}=\emptyset,$$
and the result follows by definition of $G_\mu(T)$.

$(4)$ The $i$'s in row $i$ of $T$ do not contribute for $G_\mu(T)$.

$(5)$ Note $T(\ell(\mu)+i,1)=s+i-1\le \ell(\mu)+i-1<\ell(\mu)+i$, for $  i\ge 1$.

Let $T(\ell(\mu)+1,1)= s \mbox{ even }\in [\ell(\mu)]$. Then, for $1\le r\le s\le \ell(\mu)$, the partition $\mu^{(2n-r+1)}$ gives the
shape occupied by the entries $<r$, including the empty entries of the shape $\mu$  identified with $0$,  in rows $r, r + 1, \dots , \ell(\mu)$ of $T$, and $\ell(\mu^{(2n-r+1)})=\ell(\mu)-r+1$. In particular, $\mu^{(2n-s+1)}$ gives the shape occupied by the entires $<s$ in rows $s,  \dots , \ell(\mu)$, and $\mu^{(2n-s)}$ gives the shape occupied by the entires $\le s$ in rows $s+1,  \dots , \ell(\mu), \ell(\mu)+1$.
This means, $\ell(\mu^{(2n-s+1)})=\ell(\mu)-s+1$ and
$$\ell(\mu^{(2n-s)})=\ell(\mu^{(2n-s+1)})-1+1=\ell(\mu)-s+1=\ell(\mu^{(2n-s+1)}).$$

Let $T(\ell(\mu)+1,1)= s \mbox{ odd }\in [\ell(\mu)]$. Since the partition $\nu$ is even, then  one also has $T(\ell(\mu)+2,1)= s+1$.
 Assume, for some even $2\le t\le \ell(\mu)+1$, $T(\ell(\mu)+i,1)=s+i-1$,  $ 1\le i\le t$.

 For $s=1$, one has for some even $2\le t\le \ell(\mu)+1$, $T(\ell(\mu)+i,1)=i$,  $ 1\le i\le t$. We show that this is equivalent to
 \begin{align}\ell(\mu)=\ell(\mu^{(2n)})=\ell(\mu^{(2n-1)})=\cdots=\ell(\mu^{(2n-t)}
  \end{align}

  The partition $\mu^{(2n-r+1)}$ gives the
shape occupied by the entries $<r$, including the empty entries of the shape $\mu$  identified with $0$,  in rows $r, r + 1, \dots , 2n$ of $T$, for $r = 1, 2, \dots, 2n$

For $i = 1, 2, \dots, t\le \ell(\mu)+1$ with $t$ even, since $T(\ell(\mu)+i,1)=i$,  $ 1\le i\le t$, the partition $\mu^{(2n-i+1)}$ gives the
shape occupied by the entries $<i$, including the empty entries of the shape $\mu$  identified with $0$,  in rows $i, i + 1, \dots ,\ell(\mu),\ell(\mu)+1,\dots, \ell(\mu)+i-1$ of $T$. Therefore, $\ell(\mu^{(2n-i+1)})=\ell(\mu)$,for $i=1,\dots t$.

Let $3\le s\le \ell(\mu)$ with $s$ odd. Since for some even $2\le t\le \ell(\mu)-s+2$, $T(\ell(\mu)+i,1)=s+i-1$,  $ 1\le i\le t$, then indeed
$$\ell(\mu^{(2n)})>\ell(\mu^{(2n-1)})>\cdots>
  \ell(\mu^{(2n-s+1)})=\ell(\mu)-s+1=\ell(\mu^{(2n-s)})=\cdots=\ell(\mu^{(2n-(s+t-1))})
  $$

\end{proof}

\subsection{The symplectic left Gelfand-Tsetlin pattern of a symplectic LR tableau}
Let ${}^-LRS^\lambda_{\mu,\nu}$ denote the set left companions or left Gelfand-Tsetlin patterns of the LR-Sundaram tableaux $LRS(\lambda|\mu,\nu)$ as in Definition \ref{def:lrs}.
We show that the left companions or left Gelfand-Tsetlin patterns of LR-Sundaram tableaux restrict to those satisfying the symplectic condition, Definition \ref{def:symp} and Definition \ref{def:sympgt},

\begin{align}{}^-LR^\lambda_{\mu,\nu}\cap SpT(\mu,n)={}^-LRS^\lambda_{\mu,\nu}.
\end{align}

 We do this by characterizing $LR(\lambda/\mu,\nu)\setminus LRS(\lambda/\mu,\nu)$ which is equivalent to ${}^-LR^\lambda_{\mu,\nu}\setminus {}^-LRS^\lambda_{\mu,\nu}$ by describing, for a given $T\in  LR(\lambda/\mu,\nu)\setminus LRS(\lambda/\mu,\nu)$ the topmost failing row location respectively bottom most row location for the corresponding left Gelfand-Tsetlin pattern $G_\mu(T)$.

We now recall the definition of LR-Sundaram tableau also called symplectic LR tableau in \cite{watanabe}.

\begin{defi}\cite{sundaram,sundaram90} \label{def:lrs}Let $n\in \mathbb{N}$ be given.  Let $\mu,\nu\subseteq \lambda$ such that $\ell(\mu)\le n$ and $\nu$  an even partition. A Littlewood–Richardson tableau $T$ of shape $\lambda/\mu$ and weight $\nu$ on the alphabet $[2n]$ satisfies the
\emph{Sundaram property} if for each $i = 0,\dots , \ell(\nu)/2 $, the  odd entry $2i +1$ appears in row $n+i$
or above in the Young diagram of $\lambda$. In other words, if $T(k, j) = 2i + 1$ for some cell $(k,j)$ of $T$ and $i\in\mathbb{Z}_\ge 0$, then we have $k \le n + i$.

The set of $T \in  LR(\lambda/\mu, \nu)$ satisfying
the Sundaram property is denoted by $LRS(\lambda/\mu, \nu)$ and called the set of
\emph{  LR-Sundaram tableaux} or \emph{symplectic  LR tableaux} in \cite{watanabe}.
\end{defi}

\begin{obs} Let $T\in LR(\lambda/\mu, \nu)$.
\begin{enumerate}
\item A Sundaram property violation never occurs in the first $n$ rows of $T$.

\item If $T\in LRS(\lambda/\mu, \nu)$, the possible odd numbers  in row $n+t$ of $T$ are  larger or equal than $2t+1$, for $t\ge 0$. In other words, for $t\ge 1$, the  possible odd numbers in row $n+t$ of $T$ are $2i+1$ with $t\le i$: $1$ or larger in row $n$; $3$ or larger in row $n+1$, $5$ or larger in row $n+2$, etc.


\item For $n\ge \ell(\lambda)$ any $T\in LR(\lambda/\mu, \nu)$ is LR-Sundaram.
\end{enumerate}
\end{obs}

\begin{lem}\label{lem:firstcolumn} Let $T\in LR(\lambda/\mu,\nu)$ with $\nu$ even, on the alphabet $[2n]$, and $n\ge \ell(\mu)$.
\begin{enumerate}
\item If  $T$ satisfies the Sundaram property, it holds
\begin{align}T(n+t,1)=2i+1, &\mbox{ for some $i\ge 0$ and $t\ge 1$ }\Rightarrow i\ge t\ge 1. \label{firstcolumn}\end{align}

\item If $T$ satisfies \eqref{firstcolumn},  $T(n+t,1)\ge 2t$, for all $t\ge 1$.
\end{enumerate}
\end{lem}
\begin{proof} $1)$  It is a consequence of the definition of LRS tableau with $j=1$. If $T\in LRS(\lambda/\mu,\nu)$ then, in particular, for $j=1$, \eqref{firstcolumn} holds.

$(2)$ From \eqref{firstcolumn}, indeed $T(n+1,1)\ge 2$. By  induction on  $t\ge 1$, assume $T(n+t,1)\ge 2t$. Then     either

$$T(n+t+1,1)=odd\ge 2(t+1)+1>2(t+1) \mbox{ or },$$
$$T(n+t+1,1)=even>T(n+t,1)\ge 2t\Rightarrow even\ge 2(t+1).$$
\end{proof}

The proposition below asserts that to verify the Sundaram property in an LR tableau  it is enough to check the odd entries in the rows below row $n$ in the first column of $T$.
\begin{prop}Let $T\in LR(\lambda/\mu,\nu)$ with $\nu$ even, on the alphabet $[2n]$, and $n\ge \ell(\mu)$.  Then, $T\in LRS(\lambda/\mu,\nu)$ if and only if $T$ satisfies \eqref{firstcolumn}
\begin{align}
T(n+t,1)=2i+1, &\mbox{ for some $i\ge 0$ and $t\ge 1$ }\Rightarrow i\ge t\ge 1.\nonumber 
\end{align}
  In other words $T\in LRS(\lambda/\mu,\nu)$ if and only  for all $t\ge 1$, either $T(n+t,1)=even\ge 2t$ or $T(n+t,1)=odd\ge 2t+1$.
\end{prop}

\begin{proof} The  "only if part" was proved in Lemma \ref{lem:firstcolumn}.

"If part". Assume that condition \eqref{firstcolumn} holds for $T\in LR(\lambda/\mu, \nu)$. We want to show that for $t\ge 1$, the  possible odd numbers in row $n+t$ of $T$ are $2i+1$ with  $t\le i$. Indeed, from \eqref{firstcolumn}, $T(k,1)\ge 2$ for $k\ge n+1$.

 If $T(n+t,1)=2i+1$ and $T(n+t,j)=2i'+1$ for some $j>1, t\ge 1$ then from the semistandard property of $T$, $i'\ge i\ge t$ and $i'\ge t$.

 If $T(n+t,1)$ is even and $T(n+t,j)=2i+1$ is odd for some $j>1$ and $i\ge 1$, then $T(n+t,1)<2i+1$ and $T(n+t,1)=even\le 2i$.
 From Lemma \ref{lem:firstcolumn}, $(2)$, one has $2t\le T(n+t,1)=even\le 2i$ which implies $t\le i$.
 \end{proof}

\begin{cor} \label{neglrs} Let $T\in LR(\lambda/\mu,\nu)$ with $\nu$ even, on the alphabet $[2n]$, and $n\ge \ell(\mu)$. $T$ does not satisfy Sundaram property if and only if $T(n+t,1)=2i+1$, for some  $t> i\ge 0$.

\end{cor}

A more detailed description of the Sundaram property violation on a Littlewood-Richardson tableau of even weight, comprising bad locations of even entries which report a previous failing,  is the following.
\begin{cor} \label{cor:violation} Let $T\in LR(\lambda/\mu,\nu)$ with $\nu$ even, on the alphabet $[2n]$, and $n\ge \ell(\mu)$. Then
\begin{enumerate}
\item $T(n+t,1)=even< 2t$, for some $t\ge 1$  only if $T(n+s,1)=odd< 2s+1$, for some $1\le s<t$.

\item If $s\ge 1$ is such that $(n+s,1)$ is the first cell, seen from the top, where $T(n+s,1)=odd<2s+1$ then   $$T(n+1,1)\ge 2,\dots,T(n+s-2,1)\ge 2(s-2),  T(n+s-1,1)= 2(s-1),$$ 
    and $$T(n+s,1)=2(s-1)+1, T(n+s+1,1)=2s<2(s+1).$$

\item  $T(n+s,1)=odd< 2s+1$ for some $s\ge 1$, only if $T(n+t,1)=even< 2t$, for some  $1\le s<t$.
\end{enumerate}
\end{cor}
\begin{proof} $(1)$
If $T(n+t,1)=2a$ for some $1\le a<t$, and $t\ge 2$, then there exists $1\le s<t$ such that $T(n+s,1)=odd<2s+1$.
One has only $a-1$ positive even numbers to distribute on $t-1>a-1$ cells $(n+1,1),\dots,(n+t-1,1)$. Hence, there exists at least one cell $(n+s,1)$ such that $T(n+s,1)=odd$ with $1\le s\le t-1$. Let $(n+s,1)$ be the first cell above the cell $(n+t,1)$ (seen from the bottom) where this occurs.
If $odd\ge 2s+1$, since $\nu$ is even, it follows $T(n+s+1,1)=2(s+1)<\cdots<T(n+t,1)$ and $T(n+t,1)\ge 2t>2a$ which is absurd.

$(2)$
If $s\ge 1$ is such that $(n+s,1)$ is the first cell, seen from the top, where $T(n+s,1)=odd<2s+1$ then from the previous implication, all previous cells in the first column of $T$ satisfy $T(n+1,1)\ge 2,\dots, T(n+s-1,1)\ge 2(s-1)$. Then standard-ness forces $T(n+s-1,1)=2(s-1)$, $T(n+s,1)=2(s-1)+1$ and $\nu$ even forces $T(n+s+1,1)=2s<2(s+1)$.

$(3)$ A consequence of $(2)$.
\end{proof}
\begin{ex}$(1)$ We resume  Example \ref{ex:sympviolation}: bad positions of even entries (in green) in the first column of an LR tableau with even content report a previous Sundaram property failing.

For $n=5$
\begin{align}
G_\mu(T)= \YT{0.15in}{}{
 {1},
 {\red\mathbf{2}},
 {\red\mathbf{3}},
 {\red\mathbf{4}},
 {\red\mathbf{8}},
}\in SST_{10}(1^5,0^5)\setminus SpT_{10}(1^5,0^5),\quad
 T=\YT{0.15in}{}{
 {{},{1}},
 {{},{2}},
 {{},{3}},
 {{}},
  {{}},
 {\red\mathbf{1}},
{\green\mathbf{2}},
{\green\mathbf{4}},
{\red\mathbf{5}},
{\green\mathbf{6}},
} \notin LRS(\lambda/\mu,\nu),\; \lambda=(2^3,1^7),\; \nu=(2,2,1^4,0^4).
\end{align}
$$rev(\lambda-\nu)=(0,0,1,0,0,0,1,1,1,1)$$
For $n=6$
\begin{align}
G_\mu(T)= \YT{0.15in}{}{
 {3},
 {{4}},
 {5},
 {\red\mathbf{6}},
 {{10}},
}\in SST_{12}(1^5,0^7)\setminus SpT_{12}(1^5,0^7),\quad
 T=\YT{0.15in}{}{
 {{},{1}},
 {{},{2}},
 {{},{3}},
 {{}},
  {{}},
 {{1}},
{{2}},
{{4}},
{\red\mathbf{5}},
{\green\mathbf{6}},
} \notin LRS(\lambda/\mu,\nu),\; \lambda=(2^3,1^7,0^2),\; \nu=(2,2,1^4,0^6).
\end{align}
$$rev(\lambda-\nu)=(0,0,1,0,0,0,1,1,1,1,0,0)$$
For $n=7$
\begin{align}
G_\mu(T)= \YT{0.15in}{}{
 {5},
 {{6}},
 {7},
 {{8}},
 {{12}},
}\in SpT_{14}(1^5,0^9),\quad
 T=\YT{0.15in}{}{
 {{},{1}},
 {{},{2}},
 {{},{3}},
 {{}},
  {{}},
 {{1}},
{{2}},
{{4}},
{{5}},
{{6}},
} \in LRS(\lambda/\mu,\nu),\; \lambda=(2^3,1^7,0^4),\; \nu=(2,2,1^4,0^8).
\end{align}
$(2)$ For $n=5$
\begin{align}
G_\mu(T)= \YT{0.15in}{}{
 {2},
 {{3}},
 {\red\mathbf{4}},
 {{8}},
 {{10}},
}\in SST_{10}(1^5,0^5)\setminus SpT_{10}(1^5,0^5),\quad
 T=\YT{0.15in}{}{
 {{},{1},1},
 {{},{2}},
 {{},{3}},
 {{}},
  {{}},
 {{2}},
{{4}},
{\red\mathbf{5}},
{\green\mathbf{6}},
} \notin LRS(\lambda/\mu,\nu),\; \lambda=(3,2^2,1^6,0),\; \nu=(2,2,1^4,0^4).
\end{align}
$rev(\lambda-\nu)=(0,1,1,1,0,0,0,1,0,1)$
\end{ex}
Since the symplectic property either on GT patterns or LR Sundaram tableaux  is tested in either cases by checking only the first column, it is enough to study the one-column case. If  $T\in LR(\lambda/\mu,\nu)$, with even weight and $\ell(\lambda)\le n$, has  left companion $G_\mu(T)$ with $\mu$ consisting  of a sole column  it implies that the possible columns beyond the first column of  $T$, in a row $i\in [2n]$  $T$  just have  $i$'s  which do no contribution for $G_\mu(T)$.

 On the other hand, given $T\in LR(\lambda/\mu,\nu)$ as above, if $C$ is the first column of $G_\mu(T)$ then $C$ is the left companion of the LR tableau $\widehat T\in LR(\widehat\lambda/(1^{\ell(\mu)}),\nu)$ obtained from $T$ by keeping the first column of $T$, that is, the first column of $\mu$ and the  entries $T(\ell(\mu)+1,1),\dots,T(\ell(\lambda),1)$, and adding
 $\nu_i$, $i$'s, to row $i$ in case $i$ is not an entry in the first column of $T$ which gives $\hat\lambda_i=\nu_i+1$, and $\nu_i-1$ which gives $\hat\lambda_i=\nu_i$,  otherwise.
  
  We write $\widehat T=\tilde T\cup Y(\tilde \nu)$ where $\tilde T$ is the first column of $T$ and  $Y(tilde \nu)$ is the Yamanouchi tableau of shape $\tilde \nu)$ such that $\tilde \nu$ is the content of $T$ minus the its first column. Observe that by deleting the first column of $T$ it remains still an LR tableau. That is,
  $Y(\hat \nu)$ is the rectification of $T$ minus its first column.
 Therefore $\widehat T$ has the same weight as $T$ and its one-column left companion $G_{(1^{\ell(\mu)},0^{2n-\ell(\mu)})}(\widehat T)$ is the first column of the left companion of $T$, $G_\mu(T)$. Therefore $G_{(1^{\ell(\mu)},0^{2n-\ell(\mu)})}(\widehat T)$ and thereby the first column of $G_\mu(T)$ is completely characterized by its weight $rev(\widehat \lambda-\nu)$. This give the following proposition.
 
 \begin{prop} With the above setting, the first column of $G_\mu(T)$ is completely characterized by $rev(\widehat\lambda-\nu)$.
 \end{prop}

\begin{ex}\label{reduce1column} Let $n=7$ and $\nu=(4^2,3^2,2^2, 1^6,0^2)$ even partition
\begin{align} T=\YT{0.15in}{}{
 {{},{},{},{},{1}},
 {{},{},{},{1},{2}},
 {{},{},{1},{2},{3}},
 {{},{2},{3},{4}},
 {,3,4,5},
{\mathbf{1},6},
{2,7},
{4,8},
{\mathbf{5},9},
{6,10},
{\mathbf{11}},
{12},
}\quad G_\mu(T)=  \YT{0.15in}{}{
 {5,,{},{}},
 {6,{},{}},
 {7,{}},
 {{8}},
 {{12}},
}
\end{align}
Let \begin{align}\tilde T=\YT{0.15in}{}{
 {{}},
 {{}},
 {{}},
 {{}},
 {{}},
{{1}},
{2},
{4},
{{5}},
{6},
{{11}},
{12},
}
\quad
  Y(\tilde \nu)=\YT{0.15in}{}{
 {{1},{1},1},
 {{2},{2},2},
 {{3},3,3},
 {{4},4},
 {{5}},
{{6}},
{7},
{8},
{9},
{10},
}    \quad \rightarrow \quad \widehat T=\tilde T\cup Y(\tilde\nu)=\YT{0.15in}{}{
 {{},1,1,1},
 {{},2,2,2},
 {{},3,3,3},
 {{},4,4},
 {{},5},
{{1},6},
{2,7},
{4,8},
{{5},9},
{6,10},
{{11}},
{12},
}\quad G_{(1^{\ell(\mu)},0^{2n-\ell(\mu)})}(\widehat T)=  \YT{0.15in}{}{
 {5},
 {6},
 {7},
 {{8}},
 {{12}},
}
      \end{align}
      $$rev(\widehat \lambda-\nu)=(4^3,3,2^6,1^2,0^2)-\nu=(0^4,1^4,0^3,1,0^2). $$
\end{ex}

\begin{thm}(Main result)\label{th:main} Let $T\in LR(\lambda/\mu,\nu)$ with $\nu$ even, on the alphabet $[2n]$, and $ \ell(\mu)\le n$.
 $T$ does not satisfy the Sundaram property if and only if $G_\mu(T)$ is not symplectic.

Moreover,  in this case, there exists a unique $t\ge 0$ such that
\begin{enumerate}

\item  ${n+t+1}$  is  the minimal row number of $T$ where the  Sundaram property  violation occurs.

\item
  \begin{align}&T(n+t,1)=2t,\; T(n+t+1,1)= 2t+1,\; T(n+t+2,1)=2(t+1),\\
 &T(n+1,1)\ge 2,\, T(n+2,1)\ge 4,\dots, T(n+t+1-2,1)\ge 2(t+1-2).
 \end{align}
 \item  the maximal row    of $G_\mu(T)$ where a symplectic violation  occurs is among the bottom most $t+1$ cells,  $(\ell(\mu),1)$,  $(\ell(\mu)-1,1),\dots$,  $ (\ell(\mu)-t,1) $ of the first column of $G_\mu(T)$.
\end{enumerate}
\end{thm}
\begin{proof} Let $ G=G_\mu(T)$ and recall $\ell(\lambda)-n\le n$. From Corollary \ref{neglrs}, let $t+1\ge 1$ be  minimal such that $T(n+t+1,1)=2i+1$ for some  $0\le i<t+1$.

For readability we start to spelling out the cases $t+1=1,2$.

If $\mathbf{t+1=1}$, $T(n+1,1)=1$  and $n=\ell(\mu)$. Since $\nu$ is even,  $T(n+2,1)=2$. Hence $$2\le \ell(\lambda)-n\le n\Rightarrow n=\ell(\mu)\ge 2.$$ We show that $T(n+1,1)=1$ means a symplectic violation in the cell $(\ell(\mu),1)$ of $G_\mu(T)$.

One has, $\ell(\mu)\ge 2$, and $T(\ell(\mu)+1,1)=1$,   $T(\ell(\mu)+2,1)=2$ is equivalent to
\begin{align}&\ell(\mu)=\ell(\mu^{2n})=\ell(\mu^{2n-1})=\ell(\mu^{2n-2})\ge\ell(\mu^{2n-3})\\
&\Leftrightarrow \\
&G(\ell(\mu),1)\le 2\ell(\mu)-2=2(\ell(\mu)-1)\Leftrightarrow G(\ell(\mu),1)\ngeq 2\ell(\mu)-1.
\end{align}
Hence,  $G_\mu(T)$ is not symplectic with bottom most symplectic violation in the cell $(\ell(\mu),1)$.

If $\mathbf{t+1=2}$,  by definition of $t+1$, $T(n+2,1)=3$ and $T(n+1,1)=2$. On the other hand, $\nu$ is even,  so $T(n+3,1)=4$ and $3\le \ell(\lambda)-n\le n$ and thus $n\ge 3$, and
$$T(n+1,1)=2, \;T(n+2,1)=3, T(n+3,1)=4.$$

We show that $T(n+2,1)=3$ means a symplectic violation in a cell of the first column of $G_\mu(T)$.

\medskip
\emph{Case $n=\ell(\mu)\ge 3$}:   $T(\ell(\mu)+1,1)=2$, $T(\ell(\mu)+2,1)=3$,
$T(\ell(\mu)+3,1)=4$.

This translates to

\begin{align*}&\ell(\mu)=\ell(\mu^{2n})>\ell(\mu^{2n-1})=\ell(\mu^{2n-2})=\ell(\mu^{2n-3})=\ell(\mu^{2n-4})\ge\cdots \\
&\Leftrightarrow \\
&G(\ell(\mu),1)=2n\ge 2n-1,\;\;
G(\ell(\mu)-1,1)\le 2n-4=2(\ell(\mu)-1)-2 \ngeq 2(\ell(\mu)-1)-1
\end{align*}
So
$G$ is not symplectic with bottom most symplectic violation in the cell $(\ell(\mu)-1,1)$.

\medskip

\emph{Case  $n=\ell(\mu)+1\ge 3$} in which case and $\ell(\mu)\ge 2$, $T(\ell(\mu)+1,1)=1$, $T((\ell(\mu)+1)+1,1)=2$,
$T((\ell(\mu)+1)+2,1)=3$, $T((\ell(\mu)+1)+3,1)=4$.

This translates to

\begin{align*}&\ell(\mu)=\ell(\mu^{2n})=\ell(\mu^{2n-1})=\ell(\mu^{2n-2}=\ell(\mu^{2n-3})=\ell(\mu^{2n-4}\ge\cdots\nonumber \\
&\Leftrightarrow\\
 &G(\ell(\mu),1)\le 2n-4=2(\ell(\mu)+1)-4=2\ell(\mu)-2\ngeq 2\ell(\mu)-1.
\end{align*}
Hence,  $G_\mu(T)$ is not symplectic with bottom most symplectic violation in the cell $(\ell(\mu),1)$.
\medskip

\bigskip

Le $\mathbf{t+1}\ge 1$ with $t\ge 0$, be  minimal such that $T(n+t+1,1)=2i+1$ for some  $0\le i\le t$. From Corollary \ref{cor:violation}, this means that
$T(n+k,1)$ does not violate the Sundaram property for $1\le k\le t$ and  $T(n+k,1)\ge 2k$, for $1\le k\le t$. Therefore, since $T(n+t,1)\ge 2t$, $T(n+t+1,1)\le 2t+1$ and
$$2\le T(n+1,1)<\cdots<T(n+t,1)<T(n+t+1,1)\le 2t+1,$$
it follows that $$T(n+t,1)=2t<T(n+t+1,1)= 2t+1$$

 On the other hand, $\nu$ is even,  so $T(n+t+2,1)=2(t+1)$. Thus  $\mathbf{n+t+1}\ge n+1$, with $t\ge 0$, is  the minimal row of $T$ where the  Sundaram property  violation occurs if and only if
 $$ T(n+1,1)<\cdots <T(n+t+1-2,1)<T(n+t,1)=2t,\; T(n+t+1,1)= 2t+1,\; T(n+t+2,1)=2(t+1),$$
 $$T(n+1,1)\ge 2\times 1,\, T(n+2,1)\ge 2\times 2,\dots, T(n+t+1-2,1)\ge 2(t+1-2).$$
 Moreover $t+2\le \ell(\lambda)-n\le n$ and thus $n\ge t+2$.

 \medskip

 \emph{Case $n=\ell(\mu)\ge t+1+1$}:

\medskip
Let $T(n+1,1)=2, T(n+2,1)=4,\dots, T(n+t-1),1)=2(t-1),\;$ and $T(n+t,1)=2t,\;$ $T(n+t+1,1)= 2t+1,$ $\; T(n+t+2,1)=2(t+1)$. Note $2,4,\dots, 2(t-1)$ are the first positive $t-1$ even numbers.

 This translates to
\begin{align}\label{vip}\\
&\ell(\mu)=\ell(\mu^{2n})>\nonumber\\
&\ell(\mu^{2n-1})=\ell(\mu^{2n-2})>\nonumber\\
&\ell(\mu^{2n-3})=\ell(\mu^{2n-4})>\nonumber\\
&\vdots \qquad\qquad\qquad\qquad\qquad\nonumber\\
&\ell(\mu^{2n-(2(t+1-2)-1)})=\ell(\mu^{2n-2(t+1-2)})>\nonumber
\\
&\ell(\mu^{2n-(2(t+1-2)+1)})=\ell(\mu^{2n-(2(t+1-2)+2)})=\ell(\mu^{2n-(2(t+1-2)+3)})=\ell(\mu^{2n-(2(t+1-2)+4)})\ge\cdots \label{flat} \\
&\Leftrightarrow \nonumber\\
&G(\ell(\mu),1)=2n\ge 2n-1\;,  G(\ell(\mu)-1,1)=2n-2\ge 2(n-1)-1, \dots,\nonumber\\ &G(\ell(\mu)-(t+1-2),1)=2\ell(\mu)-2(t+1-2)=2(\ell(\mu)-(t+1-2))\ge 2(\ell(\mu)-(t+1-2))-1,\nonumber\\
&G(\ell(\mu)-(t+1-1),1)\le 2n-(2(t+1-2)+4)=2\ell(\mu)-(2(t+1-2)+4)=2(\ell(\mu)-t)-2\nonumber\\
&\ngeq 2(\ell(\mu)-t)-1
\end{align}
Note $\ell(\mu^{2n-2(t+1-2)})=\ell(\mu)-(t+1-2)-1)=\ell(\mu)-(t+1-1)=\ell(\mu)-t$. Furthermore, $t$ is the number of $>$ in \eqref{vip} before arriving to the flat sequence in \eqref{flat} which is also the number of the first $t$ non negative even numbers, $t=\#\{0,2,4,\dots,2(t+1-2)\}$.

Thus $G_\mu(T)$ is not symplectic with bottom most symplectic violation in the cell $(\ell(\mu)-t,1)$.

\medskip
In the remaining  cases there exists at least one odd number among

$$T(n+1,1)\ge 2\times 1, T(n+2,1)\ge 2\times 2,\dots,
T(n+t-2),1)\ge 2\times (t-2),T(n+t-1),1)\ge 2\times (t-1).$$

This means, at least one of the even numbers $2,4,\dots, 2(t-1)$ is replaced by an odd number, that is, for some $1\le i\le t-1$, $2i$ is replaced by $2i+1$ and $2(i+1)$ is preserved in the list.
Indeed such  odd numbers do not violate the Sundaram condition. Each time we do this we glue two next flats subsequences in \eqref{vip} and reduce by one unity the number of $>$.  The  flat tail \eqref{flat} can be longer but the flat portion
$\ell(\mu^{2n-(2(t+1-2)+1)})=\ell(\mu^{2n-(2(t+1-2)+2)})=\ell(\mu^{2n-(2(t+1-2)+3)})=\ell(\mu^{2n-(2(t+1-2)+4)})$ is preserved. Therefore,  in   the first column of $G_\mu(T)$, from the bottom, the first symplectic violation   occurs among the bottom most $t+1$ cells $(\ell(\mu),1)$,  $(\ell(\mu)-1,1),\dots, (\ell(\mu)-t,1)$.

\end{proof}

\begin{cor} The Berenstein-Gelfand-Zelevinsky LR model on left Gelfand-Tsetlin patterns \cite[Theorem]{BZphy} restricts to symplectic left Gelfand-Tsetlin patterns as left companions of LR-Sundaram tableuax. In other words, the Berenstein-Gelfand-Zelevinsky LR model in \cite[Theorem 4.3]{BZphy} restricted to symplectic Gelfand-Tselin patterns is in bijection LR symplectic (LR-Sundaram) tableaux.
\end{cor}

\medskip

\subsection{Examples}
\label{ex:example}
In all examples below $T$ is an LR tableau with even weight.
\begin{enumerate}
\item Let $n=5$ and
 \begin{align} T=\YT{0.15in}{}{
 {{},{},{},{}},
 {{},{},{},{1}},
 {{},{},{1},{2}},
 {{},{2}},
 {{}},
{\mathbf{\color{red}1}},
{2},
}
\quad
T=\YT{0.15in}{}{
 {{},{},{},{},{1}},
 {{},{},{},{1},{2}},
 {{},{},{1},{2},{3}},
 {{},{2},{4}},
 {,5},
{\mathbf{\color{red}1},6},
{2},
}
\quad\mbox{or } T=\YT{0.15in}{}{
 {{},{},{},{},{1}},
 {{},{},{},{1},{2}},
 {{},{},{1},{2},{3}},
 {{},{2},{4}},
 {,5},
{\mathbf{\color{red}1},6},
{2,7},
{8}
}\quad \mbox{ and }
G_\mu(T)=  \YT{0.15in}{}{
 {4,,{},{}},
 {5,{},{}},
 {6,{}},
 {{7}},
 {\mathbf{\color{red}8}},
}
\end{align}
where just the first column of $G_\mu(T)$ is exhibited
\begin{align}\ell(\mu)=\ell(\mu^{(10)})&=5=\ell(\mu^{(10-1)})=
 \mathbf{\color{red} \ell(\mu^{(10-2)})=5}\\
 >\ell(\mu^{(10-3)})&=4\nonumber\\
 >\ell(\mu^{(10-4)})&=3\nonumber\\
 >\ell(\mu^{(10-5)})&=2\nonumber\\
 >\ell(\mu^{10-6})&=1\nonumber\\
 >\ell(\mu^{10-7})&=0\nonumber\\
 &=\ell(\mu^{10-8})=\ell(\mu^{10-9})=0\nonumber
 \end{align}

 $T$ fails the Sundaram property and $G_\mu(T)$ fails the symplectic property (in red), with $t+1=1$.
 \item Let $n=7$,

\begin{align} T=\YT{0.15in}{}{
 {{},{},{},{},{1}},
 {{},{},{},{1},{2}},
 {{},{},{1},{2},{3}},
 {{},{2},{3},{4}},
 {,3,4,5},
{\mathbf{1},6},
{2,7},
{4,8},
{\mathbf{5},9},
{6,10},
{\mathbf{11}},
{12},
}\quad G_\mu(T)=  \YT{0.15in}{}{
 {5,,{},{}},
 {6,{},{}},
 {7,{}},
 {{8}},
 {{12}},
}
\end{align}
\begin{align*}\ell(\mu)=\ell(\mu^{(14)})=5&=\ell(\mu^{(14-1)})=5=
  \ell(\mu^{(14-2)})=5\\
  >\ell(\mu^{14-3})&=4\\
  &=\ell(\mu^{14-4})=4
  =\ell(\mu^{14-5})=4=\ell(\mu^{14-6})=4\\
  >\ell(\mu^{14-7})&=3\\
  >\ell(\mu^{14-8})&=2\\
  >\ell(\mu^{14-9})&=1\\
  >\ell(\mu^{14-10})&=0\\
  &=\ell(\mu^{14-11})=\ell(\mu^{14-12})=\ell(\mu^{14-13})=0
  \end{align*}
For $n=7$, $T$ satisfies the  Sundaram property and $G$ is symplectic.
  \item For $n=5$
  \begin{align}\quad T=\YT{0.15in}{}{
 {{},{},{},{},{1}},
 {{},{},{},{1},{2}},
 {{},{},{1},{2},{3}},
 {{},{2},{3},{4}},
 {,3,4,5},
{\mathbf{\color{red}1},6},
{2,7},
{4,8},
{\mathbf{\color{red}5},9},
{6,10},
}
\quad G_\mu(T)=  \YT{0.15in}{}{
 {1,,{},{}},
 {{\color{red}2},{},{}},
 {3,{}},
 {{4}},
 {{\color{red}8}},
}
\end{align}

  \begin{align}\ell(\mu)=\ell(\mu^{(10)})=5&=\ell(\mu^{(10-1)})=
  \mathbf{\color{red}\ell(\mu^{(10-2)})=5}\\
  >\ell(\mu^{(10-3)})&=4\\
  &=\ell(\mu^{(10-4)})=\ell(\mu^{(10-5)})=\mathbf{\color{red}\ell(\mu^{10-6})=4}\\
  >\mathbf{\color{red}\ell(\mu^{10-7})=3}&\\
  >\mathbf{\color{red}\ell(\mu^{10-8})=2}&\\
  >\ell(\mu^{10-9})=1&
  \end{align}

  $T(5+1,1)=1\ngeq 2\times 1+1 =3$, $T(5+4,1)=5\ngeq 2\times 4+1$ and $T$ fails the Sundaram property with $t+1=1$, and $G_\mu(T)$ is not symplectic.

  For $n=6$,
  $T$ is not Sundaram with $t+1=3$.

  If $T(6+1,j)=odd$ for some $j$ then $odd\ge 3>2=T(6+1,1)$; $T(6+2,1)=4 even$, and if $T(6+2,j)=odd\ge 5\ge 2\times 2+1 $ for some $j$; $\mathbf{\color{red}T(6+3,1)=5\ngeq 2\times 3+1}$; $T(6+4,1)=6$, $T(6+5)=11\geq 2\times 5+1$
  and $G(T)= \YT{0.15in}{}{
 {3,,{},{}},
 {4,{},{}},
 {5,{}},
 {{\mathbf{\color{red}6}}},
 {{10}},
}
$ is not  symplectic:

 $i=4,\quad G(3,1)=5=2\times 3-1$; $\mathbf{\color{red}\mathbf{G(4,1)=6}\ngeq 2\times 4-1}$, $\quad G(5,1)=10> 2\times 5-1$.

 \begin{align}\ell(\mu)=\ell(\mu^{(12)})=5&=\ell(\mu^{(11)})=
  \ell(\mu^{(10)})\nonumber\\
  >\ell(\mu^{(9)})&=4\nonumber\\
  &=\ell(\mu^{(8)})
  =\ell(\mu^{(7)})=\mathbf{\color{red}\ell(\mu^6)=4}\\
  >\ell(\mu^5)&=3\nonumber\\
  >\ell(\mu^4)&=2\nonumber\\
  >\ell(\mu^3)&=1\nonumber\\
  >\ell(\mu^2)&=\ell(\mu^1)=0\nonumber
  \end{align}

   For $n=7$, $T$ is Sundaram and $G$ is symplectic.

 \item Let $n=5$ and  \begin{align}\quad \quad T=\YT{0.15in}{}{
 {{},{},{},{},{1}},
 {{},{},{},{1},{2}},
 {{},{},{1},{2},{3}},
 {{},{1},{3},{4}},
 {,2,4,5},
{2,3},
{4,6},
{\mathbf{\color{red}5},7},
{6,8},
}
\quad G_\mu(T)=  \YT{0.15in}{}{
 {2,,{},{}},
 {3,{},{}},
 {\mathbf{\color{red}4},{}},
 {{8}},
 {{10}},
}\end{align}

\begin{align}\ell(\mu)=\ell(\mu^{(10)})&=5\nonumber\\
>\ell(\mu^{(10-1)})&=4\nonumber\\
&=\ell(\mu^{(10-2)})=4\nonumber\\
  >\ell(\mu^{(10-3)})&=3\nonumber\\
  &=\ell(\mu^{(10-4)})=\ell(\mu^{(10-5)})=\mathbf{\color{red}\ell(\mu^{10-6})=3}\\
  >\ell(\mu^{10-7})&=2>\ell(\mu^{10-8})=1>\ell(\mu^{10-9})=0\nonumber
  \end{align}
$T$ is not symplectic with $t+1=3$.
  \end{enumerate}
\bibliography{sample17}
\bibliographystyle{alpha}
\end{document}